\documentclass[reqno,11pt]{amsart}
\usepackage{amsmath, amssymb, amsthm,amsfonts} 
\usepackage[english]{babel}
\usepackage{bbm}
\usepackage{graphicx}
\usepackage{url}
\usepackage{epstopdf}
\usepackage[a4paper,bindingoffset=0.5cm,left=1in,right=1in,top=1in,bottom=1in,footskip=.8cm]{geometry}

\usepackage{xcolor}
\usepackage{etoolbox}
\usepackage{setspace}
\usepackage{makecell}

\usepackage{verbatim}

\newbool{final}
\setbool{final}{true}

\usepackage{tikz}
\usetikzlibrary{arrows, automata,positioning,calc,shapes,decorations.pathreplacing,decorations.markings,shapes.misc,petri,topaths}
\usepackage{tkz-berge}
  \usepackage{pgfplots}
  \pgfplotsset{compat=newest}
  \usetikzlibrary{plotmarks}
  \usepackage{grffile}
\newlength\figureheight
  \newlength\figurewidth
\pgfplotsset{%
    tick label style={font=\scriptsize},
    label style={font=\footnotesize},
    legend style={font=\footnotesize},
         every axis plot/.append style={very thick}
}

\usepackage{rotating}
\usepackage{amsbsy,enumerate}
\usepackage{graphicx}
\usepackage{ccaption}
\usepackage{comment}
\usepackage{mathrsfs}

\newcommand{\vb}{\vspace{2.2mm}}
\renewcommand{\hat}{\widehat}

\newcommand{\Wopt}{W^{\text{\sc opt}}}
\newcommand{\Iopt}{I^{\text{\sc opt}}}
\newcommand{\Wsvf}{W^{\text{\svf}}}
\newcommand{\Isvf}{I^{\text{\svf}}}

\newcommand{\bern}{J}

\newcommand{\stp}[1]{X_{#1}}
\newcommand{\stpn}[1]{X_{n,#1}}

\renewcommand{\leq}{\leqslant}
\renewcommand{\geq}{\geqslant}

\renewcommand{\vec}[1]{\boldsymbol{#1}}
\newcommand{\sti}[2]{B_{#2,#1}}
\newcommand{\stv}[1]{\vec{B_{#1}}}
\newcommand{\sdev}[2]{\sigma_{#2,#1}}
\newcommand{\mean}[2]{\mu_{#2,#1}}

\allowdisplaybreaks

\makeatletter
\renewcommand{\fnum@figure}[1]{\textbf{\figurename~\thefigure}. }
\renewcommand{\fnum@table}[1]{\textbf{\tablename~\thetable}. }
\makeatother

\usepackage{lipsum}
\usepackage{graphicx}
\usepackage{amsmath,amssymb,amsthm}
\usepackage{color}
\usepackage{csquotes}
\usepackage[rightcaption]{sidecap}
\usepackage{bbm,booktabs,array,bm}

\iffinal
\usepackage{microtype}
 \newcommand{\todo}[1]{}
 \newcommand{\nnote}[1]{}
 \newcommand{\mnote}[1]{}
 \newcommand{\manote}[1]{}

\else

\newcommand{\todo}[1]{{\color{red!50!black}TODO: #1}}
\newcommand{\nnote}[1]{{\color{blue!50!black}Neil: #1}}
\newcommand{\mnote}[1]{{\color{blue!50!black}Michel: #1}}
\newcommand{\manote}[1]{{\color{blue!50!black}Madelon: #1}}

\fi

\title[The smallest-variance-first rule  in appointment sequencing]{Performance of the smallest-variance-first rule \\in appointment sequencing}
\author{Madelon A.\ de Kemp, Michel Mandjes, Neil Olver}

\theoremstyle{plain}                    
\newtheorem{theorem}{Theorem}[section]
\newtheorem{lemma}[theorem]{Lemma}
\newtheorem{proposition}[theorem]{Proposition}

\newtheorem{assumption}[theorem]{Assumption}

\theoremstyle{definition}
\newtheorem{remark}[theorem]{Remark}
\newtheorem{definition}[theorem]{Definition}
\newtheorem{example}[theorem]{Example}

\newenvironment{proofof}[1]{\begin{proof}[Proof of #1]}{\end{proof}}

\makeatletter
\newcommand{\thickhline}{%
    \noalign {\ifnum 0=`}\fi \hrule height 1.5pt
    \futurelet \reserved@a \@xhline
}
\makeatother

\theoremstyle{plain}

\newtheorem{innercustomlemma}{Lemma}
\newenvironment{customlemma}[1]
  {\renewcommand\theinnercustomlemma{#1}\innercustomlemma}
  {\endinnercustomlemma}

\newtheorem{innercustomprop}{Proposition}
\newenvironment{customprop}[1]
  {\renewcommand\theinnercustomprop{#1}\innercustomprop}
  {\endinnercustomprop}

\newcommand{\thistheoremname}{}
\newtheorem{genericthm}[theorem]{\thistheoremname}

\newcommand{\R}{\mathbb{R}} 
\newcommand{\N}{\mathbb{N}} 
\newcommand{\E}{\mathbb{E}} 
\newcommand{\Var}{\textnormal{Var}} 
\newcommand{\p}{\mathbb{P}} 
\newcommand{\dis}{\stackrel{\text{\rm d}}{=}} 
\newcommand{\cleq}{\leq_{\text{\rm cx}}} 
\newcommand{\dleq}{\leq_{\text{\rm dil}}} 
\newcommand{\e}{\mathrm{e}} 
\newcommand{\Id}{\textnormal{id}} 
\newcommand{\sym}{\mathsf{S}}

\newcommand{\svf}{{\sc svf}}
\newcommand{\Obj}[3]{C_{#3}(#1,#2)}
\newcommand{\ObjB}[4]{C_{#4}(#1,#2,#3)}

\allowdisplaybreaks

\begin{document}

\maketitle
\setstretch{1.55}

\noindent
{\sc Abstract.}
A classical problem in appointment scheduling, with applications in health care, concerns the determination of the patients' arrival times that minimize a cost function that is a weighted sum of mean waiting times and mean idle times. 
One aspect of this problem is the \emph{sequencing problem}, which focuses on ordering the patients. 
We assess the performance of the \emph{smallest-variance-first} (\svf{}) rule, which sequences patients in order of increasing variance of their service durations.
While it is known that \svf{} is not always optimal, it has been widely observed that it performs well in practice and simulation.
We provide a theoretical justification for this observation by proving, in various settings, quantitative \emph{worst-case} bounds on the ratio between the cost incurred by the \svf{} rule and the minimum attainable cost. 
We also show that, in great generality,  \svf{} is asymptotically optimal, i.e., the ratio approaches 1 as the number of patients grows large.
While evaluating policies by considering an approximation ratio is a standard approach in many algorithmic settings, 
our results appear to be the first of this type in the appointment scheduling literature.

\vb

\noindent
{\sc Subject classification.} Health care: appointment scheduling. Scheduling: stochastic appointment sequencing. Optimization: approximation algorithm.

\noindent
{\sc Area of review.} Stochastic Models.

\vb

\noindent
{\sc Affiliations.} M.\ de Kemp and M.\ Mandjes are with Korteweg-de Vries Institute, University of Amsterdam. N.\ Olver is with the Department of Econometrics \&{} Operations Research, Vrije Universiteit Amsterdam, and is also affiliated with CWI, Amsterdam.
The research for this paper is partly funded by the NWO Gravitation Programme \textsc{Networks}, Grant Number 024.002.003 (de Kemp, Mandjes) and NWO Vidi grant 016.Vidi.189.087 (Olver).

\newpage

\section{Introduction}

Setting up appointment schedules plays an important role in health care and various other domains. The main challenge lies in efficiently running the system, but at the same time providing the customers an acceptable level of service. The service level can be expressed in terms of the waiting times the customers are facing, and the system efficiency in terms of the service provider's idle time. The problem of generating an optimal schedule is generally formulated as minimizing a cost function (or simply ``cost'') that is a weighted average of the expected idle time and the expected waiting times. As most literature on this topic focuses on applications in health care, we refer throughout this paper to customers as {patients}, and to the server as the doctor.

The problem of scheduling appointments can be split into two parts: one needs to determine the amount of time scheduled for each appointment, and one needs to determine in which order the patients should arrive. These problems are usually referred to as the \emph{scheduling} problem and \emph{sequencing} problem, respectively.
This paper will primarily focus on the sequencing problem (but also includes results on the combined sequencing and scheduling problem), in a context with a single doctor seeing a sequence of patients. We impose the common assumptions that the service times of the patients form a sequence of independent random variables, while they arrive punctually at the scheduled times (which we will refer to as \emph{epochs}). 
In this setting, a variety of techniques is available that determines for a given order the optimal arrival epochs; see, e.g., \cite{Begen, Kuiper} and references therein. 
However, much less is known about the efficient  computation of  ``good'' sequences. Already for a relatively modest number of patients, the number of possible sequences is huge, thus seriously complicating the search for an optimal order. An appointment scheduling review paper from 2017~\cite{Ahmadi} states that the optimal sequencing problem is one of the main open problems in the area:
\begin{quotation}
``[...] one of the biggest challenges for future research is to find optimal (or near-optimal) solutions to more realistic appointment sequencing problems.''
\end{quotation}

A number of papers consider the sequencing (or combined sequencing and scheduling) problem and develop various stochastic programming models for it \cite{Berg,DentonGupta,Mak,Mancilla}.
However, the resulting optimization problems are very difficult to solve.
Variants of the problem have been shown to be NP-hard \cite{Kong,Mancilla}, indicating that this difficulty is inherent.

In a popular alternative approach one considers simple heuristics for the sequencing problem.
The most frequently used heuristic is to order the patients by the variance of their service times, from smallest to largest. 
Throughout this paper we refer to this sequence as the \svf{} (smallest-variance-first) sequence.
The intuition for using the \svf{} sequence is that an unusually long service time in the beginning could cause many later patients to have to wait, and the \svf{} sequence aims to reduce the risk of this occurring. 
An additional appeal lies in the fact that \svf{} is simple to implement, as it only requires knowledge of the variances of the service times.
Simulation experiments have revealed that the \svf{} rule typically performs remarkably well. 
In some cases it can be formally shown to be optimal;
for instance (imposing some distributional assumptions) in the case of two patients \cite{Gupta,Weiss}.
Recently, however, Kong et al.~\cite{Kong} provided instances showing  \svf{}  need not be optimal, even for cases involving  relatively simple service times (e.g., uniform or lognormal).
Despite the \svf{} sequence appearing promising in simulations, little is known about its theoretical performance, or of any other simple heuristic for that matter.

In this paper, we propose a new direction of research for the sequencing problem:
\emph{finding sequences that provably perform well}. 
Instead of finding an optimal sequence, such research aims at finding performance bounds on easily-computed sequences. Considering previous research, the \svf{} sequence is the obvious candidate for such an easily-computed and well-performing sequence, and will therefore be our focus.
The precise quantity of interest to us will be the \emph{ratio} of the cost of the schedule coming from the \svf{} sequence, and the cost of the schedule coming from the optimal sequence.

Our main goal in this paper is to prove upper bounds on this ratio -- known as the \emph{approximation ratio} -- in various settings. 
This direction of study is very standard in the algorithmic community when considering intractable (NP-hard) problems, for example in machine scheduling (see \cite{GuptaK,GuptaM,Skutella} and references therein).
However, it has not been studied in the appointment sequencing context.
Note that for \emph{typical} problem instances the \svf{} sequence could perform significantly better than suggested by an upper bound on the approximation ratio, as the bound must also hold for \emph{worst-case} instances.

\subsection{Main contributions}
In the first part of the paper we concentrate exclusively on the effect of the sequence, using the simplest choice of schedule:  
each patient is assigned a slot of length equal to its mean service time.
In other words, the arrival time of any patient is set equal to the sum of the  mean service times of all preceding patients.
This is certainly not the optimal solution to the scheduling problem, in the sense of minimizing the cost function introduced above,
but it has the advantage of being very simple and easily applicable, and also completely independent of the choice of tradeoff in the cost function between doctor idle time and patient waiting time. Owing to these attractive properties, this ``mean-based'' type of schedule is a commonly used approach in practice, as 
was stated in, e.g., the survey paper \cite{Ahmadi} and in \cite{Erdogan}.

We start, in Section~\ref{sec:model}, by arguing that without any restrictions on the service-time distributions, no bound on the approximation ratio of \svf{} is possible, both under mean-based schedules and under optimally-spaced schedules (i.e., schedules in which the arrival epochs are chosen so as to minimize the cost function). 
We do so by constructing an example involving only two patients, but in which the service-time distributions are rather artificial.
Then we present the ordering assumption on the service-time distributions that will be required for most of our results.
Importantly, a large class of distributions meets this  assumption, including e.g.\ the exponential distribution, but also the lognormal distributions that are frequently used in the health care context (such as the ones identified in \cite{Cayirli}).

For the mean-based scheduling rule, we have performed an extensive set of numerical experiments for exponential  and lognormal service time distributions, to gain  insight into the performance of the \svf{} rule for typical problem instances. These experiments indicate that the \svf{} sequence performs very well for these distributions, with approximation ratios being below 1.01. However, an example at the start of Section 3 indicates that there are instances that are covered by the ordering assumption where the approximation ratio is 1.52, implying that it is not possible to prove a better bound than 1.52 if this assumption is in place.

Section~\ref{sec:bounds} focuses on the mean-based scheduling rule. Under the ordering assumption we prove that the approximation ratio of \svf{} is at most 2 for symmetric service-time distributions, and at most 4 in general. In other words, we show that \emph{for all instances} (i.e., for all numbers of patients and all service-time distributions satisfying the assumption imposed) the \svf{} cost is at most  four times the optimal cost. We also consider two special cases:
\begin{itemize}
	\item Service times are evidently nonnegative, but one could consider the situation that normal distributions are used as an approximation of the actual distributions of service times. In Section~\ref{normal}, we prove that then the approximation ratio is at most $4(\sqrt{2}-1)\approx 1.6569$. While we do not believe that our result here is sharp, it indicates that the performance of \svf{} for well-behaved service-time distributions is most likely substantially better than suggested by the bounds 2 and 4 mentioned above.
	\item In Section~\ref{lognormal} we bridge the gap between the upper bound of 2 for symmetric distributions and the general upper bound of 4, by developing a method that isolates the effect of asymmetry. For the lognormal distributions fitted to real data in \cite{Cayirli}, this method results in an approximation ratio of at most 3.43.
\end{itemize}

The problem of finding the optimal sequence becomes increasingly difficult as the number of patients grows larger, and thus the need for considering heuristics such as the \svf{} rule is more important when a substantial number of patients is involved.
In Section~\ref{sec:asymptotic}, we show for mean-based schedules that as the number of patients grows large, the approximation ratio tends to 1.  This result requires only a very weak assumption on the service-time distributions (the ordering assumption is not needed here). 
The important practical implication of this result is that \svf{} is close to optimal in settings where the number of patients is substantial.

In Section~\ref{sec:combined} we shift our attention to the performance of \svf{} when optimally-spaced schedules (i.e., with arrival epochs minimizing the cost function),  rather than mean-based schedules, are being used. 
{In this setting, our numerical experiments indicate that the \svf{} sequence tends to be optimal for exponential and lognormal service time distributions; we do however present an example satisfying the ordering assumption in which \svf{} is no longer optimal.}

We proceed by proving bounds on the approximation ratio of the \svf{} rule for the combined sequencing and scheduling problem.
Here, we wish to compare a heuristic for this combined problem to the overall optimal schedule, over all possible sequences and schedules. 
Observe that the simple mean-based scheduling rule may in general lead to high cost, because waiting times could easily propagate. 
We therefore consider a simple alternative scheduling rule, suggested by Charnetski~\cite{Charnetski}: the slot assigned to a patient is equal to its mean service time, plus some multiple $\alpha$ of the standard deviation of its service time (where this $\alpha$ is optimized).
Again under some assumptions, we show that this scheduling rule, combined with the \svf{} sequencing rule, yields a cost that is (relative to the optimal cost) off by at  most a constant factor.
Because of its frequent use in health care contexts \cite{Cayirli,Klassen}, we pay special attention to 
the case of lognormally distributed service times. 
Using a slightly different scheduling heuristic (the interarrival time being a multiple of the mean service time), we find an upper bound on the approximation ratio.
Applying this result to the data in \c{C}ay{\i}rl{\i} et al.~\cite{Cayirli}, we find an upper bound of 2.90 in the case that in the cost function the waiting and idle times are equally important.

We finish Section~\ref{sec:combined} by giving an example to demonstrate that for optimally-spaced schedules the \svf{} sequence is no longer asymptotically optimal as the number of patients tends to infinity, in contrast to the result for mean-based schedules that was presented in Section~\ref{sec:asymptotic}. 

\subsection{Further related work}
We proceed by providing an account of the related literature, without attempting to give a full overview; for more extensive reviews on the appointment scheduling and sequencing literature, we refer the reader to, e.g., Ahmadi-Javid et al.~\cite{Ahmadi}, \c{C}ay{\i}rl{\i} and Veral~\cite{CayirliReview}, and Gupta and Denton~\cite{GuptaReview}.

As we mentioned, Kong et al.~\cite{Kong} gave examples showing that  \svf{} is not in general optimal.
In some very specific cases, optimality of \svf{} has been demonstrated.
For only two patients, the \svf{} sequence is optimal when the service times are both exponentially distributed or both uniformly distributed \cite{Weiss}, or more generally, 
when the two service times are comparable according to a certain convex ordering \cite{Gupta}.
For three patients, Kong et al.~\cite{Kong} find sufficient conditions for the \svf{} sequence to be optimal, when the time scheduled for each appointment is equal to the mean service time. 
(We have verified that this result can be extended to four patients using the same methods.)

Kemper et al.~\cite{Kemper} analyze a sequential optimization approach, meaning that the arrival time of a patient is optimized without taking into account its impact on later patients.
They show that under this rather different notion of optimality, and if the service times come from the same scale-family, then \svf{} does provide the best ordering.

One line of research focuses on comparing various sequencing heuristics (including \svf{}) through simulation.
Denton et al.~\cite{Denton} consider a model similar to ours, and discuss the effectiveness of a number of simple sequencing heuristics using simulation, based on real surgery data.
The \svf{} heuristic performed best of all the heuristics they considered.
Mak et al.~\cite{Mak} consider a model where waiting time costs may be different for different patients; by relying on tractable approximations, they also find that \svf{} performs well.
Klassen and Rohleder~\cite{Klassen} and Rohleder and Klassen~\cite{Rohleder} consider an appointment scheduling model where not all patient information is known in advance; rather, patients must be scheduled as they call in to make an appointment (and so without information about patients who call later).
Once again, it was empirically found that it worked best to put patients with low-variance service times early in the schedule. 

A number of papers model variants of the combined sequencing and scheduling problem as stochastic integer or linear programming problems.
Solving these programs is very challenging however, and generally exact results were only obtained for small instances.
Works along these lines include Denton and Gupta~\cite{DentonGupta}, Mancilla and Storer~\cite{Mancilla} and Berg et al.~\cite{Berg}.
For larger instances, it was necessary to resort to heuristics such as \svf{} for the sequencing problem.
We mention Vanden Bosch and Dietz~\cite{VandenBosch} who propose instead a local search heuristic to iteratively improve the sequence by finding pairs of patients who can be swapped to reduce the cost. 

There are also a number of papers which take a robust optimization approach \cite{Kong2013,Mak2,Mittal}.
Here, instead of working with explicitly given service-time distributions, the goal is to find a schedule minimizing the worst-case expected cost given only that the distributions meet certain constraints (such as certain given moments). The main advantage of robust optimization is that only the constraints are needed rather than full distributional information.
Note the difference with our approach: robust optimization minimizes the \emph{cost} of solution for the worst-case distributions satisfying the constraints, while our approach bounds the \emph{approximation ratio}. The bound on the approximation ratio indicates the performance compared to optimal for any problem instance. In contrast, a robust optimization solution might not be anywhere near optimal for a typical problem instance, but will be good for the worst-case instances. Robust optimization results and bounds on approximation ratios could thus be seen as complementary results.

Among the works on robust optimization, that by Mak et al.~\cite{Mak2} is most relevant to us, as they prove (under mild assumptions) that in their framework  \svf{} is optimal. 
In their setup, the joint distribution of the service times could be any distribution matching known moments for individual service times (e.g., means and variances).
However, the worst-case distributions corresponding to the optimal interarrival times are typically highly correlated; these results do not carry over to a setup in which independence is assumed.
Kong et al.~\cite{Kong2013} consider a model in which not only the means and variances but also the covariances are specified. However, only the scheduling problem is discussed, and sequencing is not taken into account.
Mittal et al.~\cite{Mittal} discuss another robust model, in which each service time can take any value in a given interval. 
They observe that a $(1+\epsilon)$-approximation algorithm for the combined scheduling and sequencing problem can be obtained, for any $\epsilon > 0$.

Finally, we would like to point out the relation with the area of machine scheduling (see the book by Pinedo~\cite{Pinedo} for more background). The main difference between machine scheduling and appointment scheduling is that in the former the arrival times of the jobs (being the patients in our setting) are given, while in the latter these are decision variables. The machine scheduling problem that is most closely related to the setup that we consider, can be found in Guda et al.~\cite{Guda}. In the problem considered there,  the due dates and sequence of jobs are to be decided, in order to minimize a weighted average of expected earliness and tardiness around the due dates. It is shown that the \svf{} rule is optimal, under a specific assumption on the service times of the jobs. It should be noted, though, that in the model considered in Guda et al.~\cite{Guda} all jobs are present from the start, implying that there is no idle time. Compared to our model, this greatly simplifies the evaluation of the cost function, thus facilitating finding an optimal solution.


\section{Model and preliminaries}\label{sec:model}

Throughout this paper we consider a problem instance with $n$ patients, numbered 1 up to $n$. We denote the service time of patient $i$ in this problem instance by $B_{i}$, which has mean $\mu_{i}$ and variance $\sigma_{i}^2$; we assume that $B_1,\ldots, B_n$ are independent random variables. 

As pointed out in the introduction, one should distinguish between the scheduling problem and the sequencing problem. The sequencing problem, on which we primarily focus, is to decide which patient is assigned which appointment slot (given a certain scheduling rule). The sequence is denoted by a permutation $\tau \in \sym_n$ (where $\sym_n$ denotes the set of all permutations on $\{1,\dots,n\}$).
The value $\tau(i)$ will denote the index of the patient that is assigned to appointment slot $i$. The scheduling problem is to decide the interarrival times between patients, given the sequence in which they arrive. We use $x_{j}$ to denote the interarrival time between patient $j$ and the next patient, i.e., the length of the appointment slot reserved for patient $j$. 
The vector $\bm{x} = (x_1, x_2, \ldots, x_n)$ will be referred to as the schedule. 

\subsection{Performance metrics}
We proceed by introducing the cost function we work with in this paper, which is based on the patients' waiting times and the doctor's idle times. Let $W_{i}$ denote the waiting time of the patient in appointment slot $i$. 
Let $I_{i}$ be the idle time before the start of appointment slot $i$ after the previous patient has been served. 
Given a sequence $\tau$ and interarrival times $x_{i}$, the waiting times and idle times satisfy the Lindley recursions~\cite{Lindley}:
\begin{align}\label{LindleyModel}
    W_{i+1}=(W_{i}+B_{\tau(i)}-x_{\tau(i)})^+,\:\:\:\:\: I_{i+1}=(W_{i}+B_{\tau(i)}-x_{\tau(i)})^-,
\end{align}
where we use  the compact notation $x^+ := \max\{0,x\}$ and $x^- := \max\{0,-x\}$. 

We use a parameter $\omega\in(0,1)$ to indicate the relative importance of idle time and waiting time. As a cost function, we seek to minimize
\begin{equation}\label{objective}
    \Obj{\tau}{\bm{x}}{\omega} := \omega \sum_{i=1}^n \E I_{i} +(1-\omega)\sum_{i=1}^n \E W_{i},
\end{equation}
a weighted average of the expected total idle time and expected total waiting time. 
Observe that this cost function depends on the sequence $\tau$, on the schedule $\bm{x}$, and on the patient service-time distributions $\vec{B} = (B_1, \ldots, B_n)$.
We generally suppress the dependence on $\vec{B}$, but we may write $\ObjB{\vec{B}}{\tau}{\bm{x}}{\omega}$ if we wish to be explicit.
As an aside, we mention that an approach to estimate $\omega$ in a practical context can be found in \cite{Robinson}.

Throughout this paper, without loss of generality we let the patients be indexed such that $\sigma_{1}^2\leq \sigma_{2}^2\leq \dots \leq \sigma_{n}^2$. 
The \svf{} sequence is then the sequence given by the identity permutation $\Id$ given by $\Id(i)=i$, which we compare with the optimal sequence, i.e., the sequence that minimizes (\ref{objective}). 
To compare these sequences, we study the ratio between the cost functions under the \svf{} sequence and the optimal sequence. If this ratio is close to 1, then this is evidence that the \svf{} sequence performs well. 

In this paper we specifically focus on two settings. 
In the first place we consider the performance of \svf{} for {\it mean-based schedules}, which are given by $\bm{x}=\bm{\mu}$. We then consider the approximation ratio
\begin{align*}
    \varrho_{\omega}(\vec{B}) := \frac{\ObjB{\vec{B}}{\Id}{\bm{\mu}}{\omega}}{\min\{\ObjB{\vec{B}}{\tau}{\bm{\mu}}{\omega}: \tau \in \sym_n\}}.
\end{align*}
We will write just $\varrho_{\omega}$ when the service-time distributions under consideration are unambiguous.
In the second place, we consider the performance of \svf{} for {\it optimally-spaced schedules}. In this context we compare the \svf{} sequence $\Id$ along with a given schedule $\bm{{x}}$ with the optimal combination of sequence and schedule (i.e.,  a sequencing $\tau$ and schedule  ${\bm{{y}}}$ that minimize the cost function). This means that here we consider the approximation ratio
\begin{align*}
    r_{\omega}(\vec{B},\vec{{x}}) := \frac{\ObjB{\vec{B}}{\Id}{\bm{{x}}}{\omega}}{\min\{\ObjB{\vec{B}}{\tau}{\bm{y}}{\omega}:\tau \in \sym_n,\bm{y} \in \mathbb{R}_+^n\}}.
\end{align*}
Once again, we will omit $\vec{{x}}$ and $\vec{B}$ when their choice is unambiguous.

The main objective of this paper is to prove, under specific assumptions, upper bounds on $\varrho_{\omega}$ and $r_{\omega}$. 
Such an upper bound then guarantees that the \svf{} sequence always has a cost function of at most such an upper bound times the optimal cost function. We also show, under a mild condition, that $\varrho_{\omega}(\vec{B})$ converges to 1 as the number of patients tends to infinity, thus proving that the \svf{} sequence is asymptotically optimal when mean-based schedules are used.

\begin{remark}
Service times are inherently nonnegative, but our framework (based on the Lindley recursions (\ref{LindleyModel})) carries over to situations where the $B_i$ are allowed to take negative values.
This might be useful if the true distributions of service times can be approximated using distributions that can take negative values (with some small probability), for example normal distributions. If such distributions that can take negative values form a good fit to the data in some application, the theoretical performance of the \svf{} rule for these distributions gives a meaningful indication for the performance of the \svf{} rule in this application.
\end{remark}

\subsection{Preliminaries}
We need the following well-known results concerning the waiting and idle times. It follows by iterating the Lindley recursion (\ref{LindleyModel}) that $W_{k+1}$ is the maximum of a random walk with steps $B_{\tau(k)}-x_{\tau(k)}, B_{\tau(k-1)}-x_{\tau(k-1)},\dots, B_{\tau(1)}-x_{\tau(1)}$, that is,
\begin{align}\label{W}
W_{k+1} = \max\left\{0, \max\left\{\sum_{i=j}^k B_{\tau(i)}-x_{\tau(i)}:  j\in\{1,\ldots,k\} \right\}\right\}.
\end{align}
In the setting of mean-based schedules $\vec{x} = \vec{\mu}$, we introduce the notation $X_{i}:= B_{i}-\mu_{i}$, and the random walk
\begin{align*}
S_{j}:= \sum_{i=1}^j X_{\tau(k-i+1)}.
\end{align*}
Note that $S_j$ also depends on $k$ and $\tau$, but we keep this dependence implicit to ease the notation.
We then find for the mean-based schedule that
\begin{align}\label{Wmean}
W_{k+1}=\max\{0,S_{1},\dots,S_{k}\}.
\end{align}

Note that (\ref{LindleyModel}) implies $W_{i+1}-I_{i+1}=W_i+B_{\tau(i)}-x_{\tau(i)}$. Summing over $i$ we find the identity
\begin{align}\label{I}
\sum_{i=1}^n I_{i} +\sum_{i=1}^n B_{i} =\sum_{i=1}^{n-1} x_{\tau(i)} +W_{n}+B_{\tau(n)},
\end{align}
which corresponds to the total time until all $n$ patients have been served (the ``session end time'').
For a given schedule, this relation can be used to express the expected total idle time in terms of the expected waiting time of the last patient. Therefore, we can focus on the waiting times, and derive results for the idle time from \eqref{I}.

The following example shows that, if one does not impose any conditions on the service-time distributions, one can construct examples in which the approximation ratio is unbounded.

\begin{example}\label{noconstbound}
Suppose we have two patients, and the service time of patient $i$ is given by
\begin{equation}\label{counterexample}
B_i=\begin{cases}
\mu_i+\frac{1}{a_i} &\text{ with probability }a_i^2 \\
\mu_i-\frac{1}{a_i} &\text{ with probability }a_i^2 \\
\mu_i&\text{ with probability }1-2a_i^2,
\end{cases}
\end{equation}
for some values $\mu_i>0$ and $a_i\leq 1/\sqrt2$. Then $\E B_i=\mu_i$, and $\Var B_i=2$, and so either of the two possible sequences could be considered the \svf{} sequence. 
We take the \svf{} sequence to be given by $\tau(i)=i$; one can of course slightly perturb the distributions so that this is the {\it unique} \svf{} ordering. 

Suppose we use the mean-based schedule given by $\bm{x}=\bm{\mu}$. The cost function for the \svf{} sequence $\tau(i)=i$ is then
\begin{align*}
\omega \E I_2+(1-\omega) \E W_2= \omega \E(B_1-\mu_1)^-+(1-\omega)\E(B_1-\mu_1)^+= \omega a_1+(1-\omega)a_1=a_1.
\end{align*}
In the same way, the cost function for the other sequence is equal to $a_2$. If we take $a_1$ to be bigger than $a_2$, we conclude $\varrho_{\omega} =a_1/a_2$, which can be arbitrarily large. As a consequence, it is necessary to impose assumptions on the service-time distributions in order to bound~$\varrho_\omega$.
The construction can easily be extended to one corresponding to any larger number of patients, by introducing additional patients with deterministic service times. Therefore, this example also shows that some assumption is necessary for the asymptotic result in Section~\ref{sec:asymptotic}.

The example applies also when an optimally-spaced, rather than mean-based, scheduling rule is used.
Fixing $\omega=\frac12$, for two patients with service times as in \eqref{counterexample}, the cost function is 
\begin{align*}
\tfrac12 \E\left(B_{\tau(1)}-x_{\tau(1)}\right)^-+\tfrac12 \E\left(B_{\tau(1)}-x_{\tau(1)}\right)^+.
\end{align*}
By Lemma~\ref{quantile}, the $x_{\tau(1)}$ that minimizes this cost function is given by $x_{\tau(1)}=\mu_{\tau(1)}$, which is the mean of $B_{\tau(1)}$. 
So the situation is unchanged, and also for the optimal scheduling rule no bound on the approximation ratio can be found without imposing further assumptions.
\end{example}

Example~\ref{noconstbound} shows that it will be necessary to impose assumptions on the service-time distributions in order to be able to bound the approximation ratio. Notice that the three-point distributions used in the example can be considered as artificial, in the sense that they substantially differ from distributions that are frequently used in the context of health care. One would like to use an assumption under which the approximation ratio can be bounded, but which does not rule out relevant, commonly used service-time distributions. 

Not all service time distributions might be sufficiently comparable through their variance alone. We will need the concepts of a \emph{convex ordering} and \emph{dilation ordering} on random variables. More background on the convex ordering, dilation ordering and related concepts can be found in \cite{Shaked}.

\begin{definition}
The random variable $A$ is said to be smaller in the convex order than the random variable $B$ if $\E \phi(A)\leq \E \phi(B)$ for all convex functions $\phi:\R\to\R$ for which the expectations exist. This will be denoted by $A\cleq B$. If $A-\E A\cleq B-\E B$, then $A$ is said to be smaller than $B$ in the dilation order, denoted as $A\dleq B$.
\end{definition}

Note that $A\cleq B$ implies $A\dleq B$, and $A\dleq B$ implies $\Var A\leq \Var B$. Throughout this paper, unless stated otherwise, we will make the following assumption on the service time distributions.

\begin{assumption}[\bf dilation ordering]\label{AStochDom}
We have $B_1\dleq B_2\dleq\dots\dleq B_n$.
\end{assumption}

We remark that this is the condition under which \cite{Gupta} proves optimality of \svf{} for two patients. Note also that this assumption implies $\Var(B_1)\leq \dots\leq \Var(B_n)$.
Examples of instances satisfying this assumption include all $B_i$ having exponential distributions (by Theorem 3.A.18 in \cite{Shaked}), and all $B_i$ having lognormal distributions such that both $\E[ \ln B_1]\leq \dots \leq \E[\ln B_n]$ and $\Var( \ln B_1)\leq \dots\leq \Var( \ln B_n)$, as proved in Appendix~\ref{lognormalDil}. 
The following lemma, taken from \cite{Shaked}, is useful when checking whether given random variables satisfy a convex order.

\begin{lemma}\label{condExp}
The random variables $A$ and $B$ satisfy $A\cleq B$ if and only if there exists a coupling $\hat{A}\dis A$ and $\hat{B}\dis B$ such that $\E\big[\hat{B}\big|\hat{A}\,\big]=\hat{A}$.
\end{lemma}


\section{Bounds on performance under mean-based schedules}\label{sec:bounds}

In this section we study $\varrho_{\omega}$, the approximation ratio under the mean-based schedule. To gain some insight as to what approximation ratio we can expect for typical problem instances, we start by discussing the output of a series of numerical experiments, described in more detail in Appendix~\ref{MBNumerics}. It should be kept in mind that, as finding the optimal sequences in principle requires comparing all $n!$ candidates and is therefore computationally very expensive, such experiments can only be done for relatively small values of $n$. We recall that this is also the main reason why we study approximation ratios: computing the optimal sequence for larger instances is prohibitively slow, which explains why  it is useful to gain insight into the performance of heuristics.

We considered exponential and lognormal service-time distributions. The exponential distributions have the advantage of an efficiently computable cost function using the machinery developed in \cite{Wang}. For exponentially distributed service times the main conclusion is that  the \svf{} sequence typically performs within $1\%$ of optimal ($\varrho_\omega\leq 1.01$). Lognormal distributions are more realistic from a practical perspective \cite{Cayirli,Klassen}. No exact computational scheme being available, we resorted to a discrete-time approximation to evaluate the cost; see Appendix~\ref{MBNumerics} for more background on the implementation. In the experiments we performed the approximation ratios were uniformly below 1.01.

The numerics discussed in the previous paragraph suggest  that \svf{} typically performs very well.
However, the service-time distributions featuring in the next example lead to a much higher approximation ratio.

\begin{example}\label{BadExampleMB}
Suppose the service time of patient $i$ is given by
\begin{equation*}
B_i=\begin{cases} m+1 & \text{with probability } m/(m+1) \\
0 & \text{with probability }1/(m+1),
\end{cases}
\end{equation*}
for $i<n-1$, and $B_{n-1}\dis B_n \dis K B_1$, for some parameters $m$ and $K$. We set $\omega=1$. 
We compare the \svf{} sequence with the sequence  $\tau$ given by $\tau(1)=n-1$, $\tau(n-1)=1$ and $\tau(i)=i$ for $i\neq 1,n-1$.
For $n=10$ patients, $m=10$ and $K=10$ this results on a lower bound of $1.29$ on the approximation ratio $\varrho_1$. We can, however, obtain considerably higher approximation ratios: for example for $n=50\,000$ patients, $m=500$ and $K=5\,000$ we find that the approximation ratio is larger than $1.52$.
Further experiments revealed that similar examples for even bigger values of $n$ only marginally increase this lower bound.
Note that this problem instance satisfies Assumption~\ref{AStochDom}, the dilation ordering assumption. This means that it is impossible to prove a better bound on $\varrho_\omega$ than $1.52$ without imposing further assumptions. 
\end{example}

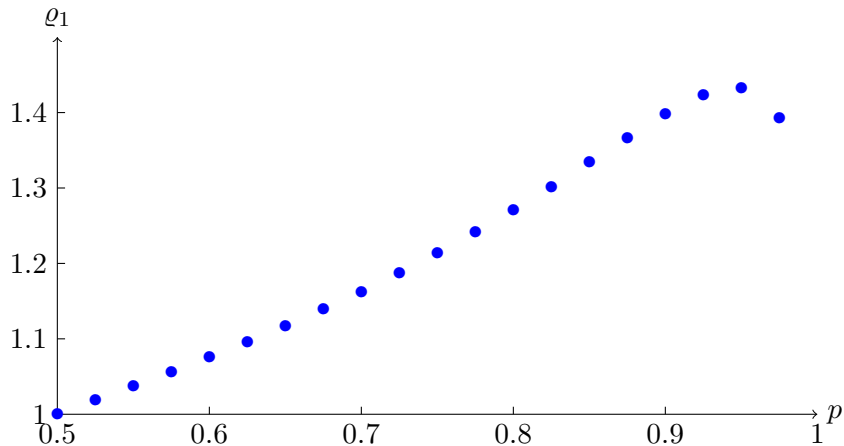
\begin{figure}[h]
\centering
\begin{tikzpicture}[xscale=20,yscale=10]
  \def\xmin{0.5}
  \def\xmax{1}
  \def\ymin{1}
  \def\ymax{1.5}
    \draw[->] (\xmin,\ymin) -- (\xmax,\ymin) node[right] {$p$} ;
    \draw[->] (\xmin,\ymin) -- (\xmin,\ymax) node[above] {$\varrho_1$} ;
    \foreach \x in {0.5,0.6,0.7,0.8,0.9,1}
    \node at (\x,\ymin) [below] {\x};
    \foreach \y in {1,1.1,1.2,1.3,1.4}
    \node at (0.5,\y) [left] {\y};
    \draw (0.6,1)--(0.6,1.01);
     \draw (0.7,1)--(0.7,1.01);
      \draw (0.8,1)--(0.8,1.01);
       \draw (0.9,1)--(0.9,1.01);
       \draw(0.5,1.1) --(0.5051,1.1);
  \draw(0.5,1.2) --(0.5051,1.2);
  \draw(0.5,1.3) --(0.5051,1.3);
  \draw(0.5,1.4) --(0.5051,1.4);

    \foreach \Point in {( 0.5 , 1 ), ( 0.525 , 1.01789 ), ( 0.55 , 1.03623 ), ( 0.575 , 1.0551 ), ( 0.6 , 1.07458 ), ( 0.625 , 1.0949 ), ( 0.65 , 1.11607 ), ( 0.675 , 1.13827 ), ( 0.7 , 1.16177 ), ( 0.725 , 1.18658 ), ( 0.75 , 1.21259 ), ( 0.775 , 1.24064 ), ( 0.8 , 1.2699 ), ( 0.825 , 1.3011 ), ( 0.85 , 1.33336 ), ( 0.875 , 1.36593 ), ( 0.9 , 1.39713 ), ( 0.925 , 1.42269 ), ( 0.95 , 1.43188 ), ( 0.975 , 1.39149 )}
    {\node[blue] at \Point {\textbullet};}
    
\end{tikzpicture}
\caption{The approximation ratio $\varrho_1$ as a function of the {parameter $p$ of} Bernoulli random variables. Each instance consists of 200 patients, with 198 of them having a Bernoulli distributed service time with the probability to take the larger value being equal to $p$, whereas the other two patients have a service time distributed as 50 times the same Bernoulli distribution.}
\label{fig:BiasRatio}
\end{figure}

The distributions in Example~\ref{BadExampleMB} could be considered as quite extreme: they are two point distribution that almost always take a value slightly larger than the mean, and very rarely take a value much less than the mean. Our numerics indicate that this behavior seems to be necessary to obtain a large gap between the performance of \svf{} and the optimal sequence. To further study the effect of such ``skewness'', we performed the following experiment. We evaluate the approximation ratio $\varrho_1$ for instances with $n=200$ patients, with $n-2$ of them having a Bernoulli distributed service time with the probability to take the larger value being equal to $p$, whereas the other two patients have a service time distributed as 50 times the same Bernoulli distribution. Figure~\ref{fig:BiasRatio} shows the approximation ratio $\varrho_1$  as a function of $p$.  In this case we are able to compute the optimal sequence, despite the relatively large value of $n$,  due to the fact that many patients have the same service-time distribution, so that the number of possible distinct sequences is drastically reduced. It can be seen that \svf{} is optimal for $p \leq 1/2$, and then grows as $p$ increases, until it decreases again for $p$ very close to one.

The above numerics suggest that \svf{} performs very well with both exponential and lognormal distributions.
These distributions have tails only to the right, and \emph{not} to the left.
If one ``flips'' the lognormal distributions, i.e., replace each service time $B_i$ with $-B_i$, we can obtain substantially larger approximation ratios; for one particular instance, we find 1.0345 as opposed to 1.0086 (see Appendix~\ref{MBNumerics} for more details). These examples suggest that some measure of skewness is expected to be a key driver in whether \svf{} performs near-optimally. 
As this experiment intends to investigate the effect of skewness on the approximation ratio we have not pursued using maximally realistic service times. Indeed, $-B_i$ is not a realistic service time as it is always negative. However, we can also use $c-B_i$ instead of $-B_i$ without affecting the cost per sequence, and choose $c$ so large that $c-B_i$ will be positive with arbitrarily large probability. Neglecting the probability mass $c-B_i$ has on the negative numbers will not have any substantial effect on the cost.

Now that we have developed some intuition on the performance of \svf{}, we concentrate in
the rest of this section on obtaining upper bounds for $\varrho_{\omega}$. This amounts to giving an upper bound on the cost  when using the \svf{} sequence, and a lower bound on the cost  that is valid for any sequence, hence also for the optimal sequence. The waiting times and idle times under the \svf{} sequence will be denoted by $\Wsvf_i$ and $\Isvf_i$ respectively, and the waiting times and idle times under the optimal sequence will be denoted by $\Wopt_i$ and $\Iopt_i$. 

The remainder of this section is structured as follows. In Section \ref{main3} we prove results under rather general assumptions: Theorem~\ref{generalComp} and Theorem~\ref{nonsymmComp}. More specifically, these theorems give bounds on the approximation ratio $\varrho_{\omega}$, when we assume that the service times are symmetrically distributed and follow a dilation order (Theorem~\ref{generalComp}), and when we only assume that they follow a dilation order (Theorem~\ref{nonsymmComp}). 
In Section~\ref{normal} we consider the special case of normally distributed service times. Theorem~\ref{thm:normal} gives an improved bound on $\varrho_{\omega}$ in this case. 
In Section \ref{lognormal} we discuss a method for improving numerically upon the bound of Theorem~\ref{nonsymmComp}; informally, the more symmetric the {service-time distributions}, the closer the resulting bound is to the value stated in Theorem~\ref{generalComp}.
Table~\ref{TableResults} provides a summary of the results covered by this section.

\begin{table}
    \centering
        \renewcommand{\arraystretch}{1.2}
    \begin{tabular}{|c|c|c|}
        \hline
        Theorem & Assumptions on service time distributions & Bound on $\varrho_\omega$ \\ 
	\Xhline{4\arrayrulewidth}
        \ref{generalComp} & dilation ordering and symmetry & 2\\ \hline
        \ref{nonsymmComp} & dilation ordering & 4\\ \hline
        \ref{thm:normal} & normal distributions & $4(\sqrt2-1)\approx 1.66$ \\ \hline
        \ref{betterThan4} & dilation ordering & between 2 and 4 \\ \hline
    \end{tabular}
    \caption{Summary of main results of Section~\ref{sec:bounds}}\label{TableResults}
\end{table}


\subsection{General results}\label{main3}

Recall that we impose Assumption \ref{AStochDom} (the dilation ordering assumption). 
For the bound on $\varrho_{\omega}$ that is presented in Theorem \ref{generalComp} we also make the following assumption. 

\begin{assumption}[\bf symmetry]\label{ASymm}
The $B_i$ have symmetric distributions around their mean.
\end{assumption}

Examples of instances satisfying both the ordering and symmetry assumption include all $B_i$ having normal distributions, all $B_i$ having uniform distributions and all $B_i$ having Laplace distributions. For all three examples, the ordering assumption follows from  \cite{Shaked}, Theorem 3.A.18. Note that Example~\ref{noconstbound} consists of symmetric service time distributions, so even under the symmetry assumption it is necessary to impose the dilation ordering assumption in order to obtain any bound on $\varrho_\omega$.

In this section, we prove the following theorems. Recall that Example \ref{BadExampleMB} discusses an instance satisfying the dilation ordering assumption where $\varrho_\omega\geq 1.52$.

\begin{theorem}\label{generalComp}
    Under the dilation ordering and symmetry assumptions, we have $\varrho_{\omega}\leq 2$.
\end{theorem}

\begin{theorem}\label{nonsymmComp}
    Under the dilation ordering assumption, we have $\varrho_{\omega}\leq 4$.
\end{theorem}

A first key insight is that it suffices to prove bounds on $\E W_{k+1}$ for given $k$ when the first $k$ slots are constrained to contain patients $1,\dots,k$. This is made explicit in the next lemma.

\begin{lemma}\label{OneWaitingTime}
    Let $\E W^{\text{\sc opt}'}_{k+1}$ denote the expected waiting time of the patient in appointment slot $k+1$, under the sequence that minimizes this expected waiting time, subject to the constraint that $\tau(i)\leq k$ for all $i=1,\dots,k$, i.e.\ the first $k$ patients are assigned to the first $k$ slots. Suppose $\E \Wsvf_{k+1}/\E W^{\text{\sc opt}'}_{k+1}\leq \varrho'$ for all $k$. Then, under the dilation ordering assumption, $\varrho_{\omega}\leq \varrho'$.
\end{lemma}

\begin{proof}
Taking expectations in \eqref{I} and using that $x_i=\mu_i$, we find that $\sum_{i=1}^n \E I_i =\E W_n.$ 
Hence,  the cost function \eqref{objective} equals
$\omega\E W_n +(1-\omega)\sum_{i=1}^n \E W_i$.
Our goal is thus to bound the ratio 
\begin{equation}\label{eq:ratioobj}
    \varrho_{\omega} =\frac{\omega\E \Wsvf_n +(1-\omega)\sum_{i=1}^n \E \Wsvf_i}{\omega\E \Wopt_n +(1-\omega)\sum_{i=1}^n \E \Wopt_i}.
\end{equation}

Now note that $W_{k+1}=\max\{0,S_1,\dots,S_k\}$ is a convex function in each of the $X_i$, as it is the maximum of linear functions in $X_i$. Under Assumption~\ref{AStochDom} this implies that $\E \Wopt_{k+1}\geq \E W^{\text{\sc opt}'}_{k+1}$, as each step $X_i$ with $i>k$ can be replaced by some $X_j$ with $j\leq k$ and as $X_j\leq_{\text{cx}}X_i$ this lowers the expected waiting time.
Now 
\begin{align*}
\E\Wsvf_{k+1}/\E\Wopt_{k+1}\leq \E\Wsvf_{k+1}/\E W^{\text{\sc opt}'}_{k+1}\leq \varrho',
\end{align*}
and so also the ratio in (\ref{eq:ratioobj}) is bounded by $\varrho'$, which completes the proof.
\end{proof}
The following lemma is a second key insight to be used in the proofs of both Theorem~\ref{generalComp} and Theorem \ref{nonsymmComp}.
\begin{lemma}\label{reflectedUB}
Under the symmetry assumption, the random variable $W_{k+1}$ is stochastically dominated by $|S_k|$, and thus
\begin{equation*}
    \E W_{k+1}\leq  \E |S_k|=\E|\stp{\tau(1)} + \stp{\tau(2)} + \cdots + \stp{\tau(k)}|.
\end{equation*}
\end{lemma}

\begin{proof}
Recall that we have $W_{k+1}=\max\{0,S_1,\dots,S_k\}$ from (\ref{Wmean}).
Under the symmetry assumption, the steps $\stp{i}$ of the random walk $S$ have a symmetric distribution around zero, and hence the same is true for the $S_i$.

Let $T(a)=\inf\{i: S_i\geq  a\}$, and note that $\p(W_{k+1}\geq  a)= \p(T(a)\leq  k)$. 
To bound this probability, we  look at the random walk reflected in $a$ after $T(a)$. This reflected process $\hat{S}_i$ is defined by
\begin{equation}\label{reflected}
\hat{S}_i =\begin{cases} S_i & \text{ if } i<T(a) \\
2a-S_i & \text{ if }i\geq  T(a).
\end{cases}
\end{equation}

We have $S_{T(a)}\geq  a$, so $\hat{S}_{T(a)}=2a-S_{T(a)}\leq  a\leq  S_{T(a)}$. As the $\stp{i}$ have symmetric distributions, the increments of $S_i$ and $\hat{S}_i$ for $i\geq  T(a)$ have the same distribution. Therefore, we see that $\hat{S}_i$ is stochastically dominated by $S_i$, for every $i$. We conclude that $\p(\hat{S}_k>a)\leq  \p(S_k>a)$ for all $a$.

Now note that $W_{k+1}\geq  a$ implies that either $S_k\geq  a$ or $\hat{S}_k=2a-S_k>a$. As these are disjoint events we now have
\begin{align*}
&\p(W_{k+1}\geq  a)=\p(S_k\geq  a)+\p(\hat{S}_k>a) \leq  \p(S_k\geq  a)+\p(S_k>a)\leq  \p(|S_k|\geq  a).
\end{align*}
This holds for any $a\geq  0$, so $W_{k+1}$ is stochastically dominated by $|S_k|$, as was claimed.
\end{proof}

\begin{proofof}{Theorem \ref{generalComp}.}
As $W_{k+1}=\max\{0,S_1,\dots, S_k\}$, we have 
\begin{align*}
W_{k+1}\geq S_k^+=(X_{\tau(1)}+\dots+X_{\tau(k)})^+.
\end{align*}
Note that $\tau(i)\leq k$ for all $i\leq k$ when we consider $\E W^{\text{\sc opt}'}_{k+1}$, so now
\begin{align}\label{convexLB}
\E W^{\text{\sc opt}'}_{k+1}\geq \E(X_1+\dots+X_k)^+.
\end{align}
On the other hand, by Lemma \ref{reflectedUB},
\begin{equation*}
    \E \Wsvf_{k+1} \leq  \E|\stp{1} + \cdots + \stp{k}| = 2\E(\stp{1} + \cdots + \stp{k})^+ \leq 2 \E W^{\text{\sc opt}'}_{k+1}.
\end{equation*}
As $\E \Wsvf_{k+1}/\E W^{\text{\sc opt}'}_{k+1}$ is now bounded by 2, Theorem~\ref{generalComp} follows from Lemma~\ref{OneWaitingTime}.
\end{proofof}

\begin{proofof}{Theorem~\ref{nonsymmComp}.}
Note that Lemma~\ref{OneWaitingTime} and the lower bound (\ref{convexLB}) are valid without the symmetry assumption being needed. We therefore only need an upper bound on $\E\Wsvf_{k+1}$.
 
Let $X_1',X_2',\dots,X_n'$ have the same distributions as respectively $X_1,X_2,\dots,X_n$ such that all these random variables are independent. Let $W_{k+1}'$  be the maximum of the random walk with steps $X_k-X_{k}', X_{k-1} -X_{k-1}', \dots, X_1-X_{1}'$. As 
\begin{align*}
\E[X_i-X_{i}'|X_i]=X_i, 
\end{align*}
we see using Lemma~\ref{condExp} that $X_i\cleq X_i-X_{i}'$. 
Note that $W_{k+1}=\max\{0,S_1,\dots,S_k\}$ is a convex function in $X_i$, as it is the maximum of functions linear in $X_i$.
Therefore, each time we replace a step $X_i$ with a step $X_i-X_{i}'$ the expected maximum of the random walk will increase, so $\E \Wsvf_{k+1}\leq \E W_{k+1}'$.

Now note that the steps $X_i-X_{i}'$ all have a symmetric distribution, so we can apply Lemma~\ref{reflectedUB} to find
\begin{align*}
\E \Wsvf_{k+1}&\leq \E W_{k+1}'\leq \E|X_1+\dots+X_k-(X_1'+\dots+X_k')|\\
&\leq 2\E|X_1+\dots+X_k| =4\E(X_1+\dots+X_k)^+\leq 4\E W^{\text{\sc opt}'}_{k+1}.
\end{align*}
As $\E \Wsvf_{k+1}/\E W^{\text{\sc opt}'}_{k+1}$ is now bounded by 4, the result follows from Lemma~\ref{OneWaitingTime}. 
\end{proofof}

\begin{remark}\label{REM1}
    In case the scheduled session end time equals the expected total service time, the 
overtime reads
\begin{align*}
W_{n+1}=\left(W_n+\stp{\tau(n)}\right)^+,
\end{align*}
which can also be included in the cost function. 
As such, overtime is handled similarly to waiting time, and consequently the results of Theorems~\ref{generalComp} and \ref{nonsymmComp} remain valid when some extra term $c\,\E W_{n+1}$ with $c>0$ is added to the cost function.
\end{remark}


\subsection{Normally distributed service times}\label{normal}

The results of Theorems \ref{generalComp} and \ref{nonsymmComp} can be strengthened for specific service-time distributions. One such result is the following.

\begin{theorem}\label{thm:normal}
    When the $B_i$ are all normally distributed we have $\varrho_{\omega}\leq 4(\sqrt2-1)$.
\end{theorem}

In order to prove Theorem~\ref{thm:normal}, we need the following two lemmas, giving stronger bounds on $\E\Wsvf_{k+1}$ and $\E W^{\text{\sc opt}'}_{k+1}$. The proofs of these lemmas, that hold for any symmetrically distributed service times, can be found in Appendix~\ref{App3}.

\begin{lemma}\label{UB}
Under the symmetry assumption, 
\begin{equation*}
\E W_{k+1} \leq  \E\left(\stp{\tau(1)} + \cdots + \stp{\tau(k)} \right)^+
                  \;+\; \E\left(\stp{\tau(1)} + \cdots + \stp{\tau(k-1)}\right)^+.
\end{equation*}
\end{lemma}

\begin{lemma}\label{LB}
Under the symmetry assumption, for any $\ell$,
\begin{align*}
\E W_{k+1} \geq  \frac12\bigg(\E\left(\stp{\tau(1)} + \cdots + \stp{\tau(k)}\right)^+ 
+\: \E\left(\stp{\tau(1)} + \cdots + \stp{\tau(\ell)}\right)^+ \:\\
+\:\E\left(\stp{\tau(\ell+1)} + \cdots + \stp{\tau(k)}\right)^+\bigg).
\end{align*}
\end{lemma}

\begin{proofof}{Theorem \ref{thm:normal}.}
Note that normal distributions satisfy both the dilation ordering and symmetry assumption.
Now the sum $\stp{1} + \cdots + \stp{i}$ again has a normal distribution, with mean zero and variance $\Sigma_i^2:= \sigma_1^2+\dots+\sigma_i^2$.
For the \svf{} sequence we now have, using Lemma~\ref{UB}, that
\begin{equation}\label{eq:normalub}
\E \Wsvf_{k+1} \leq  \frac{1}{\sqrt{2\pi}}\left(\Sigma_k+ \Sigma_{k-1}\right).
\end{equation}

Now we still need an expression for a lower bound on $\E W^{\text{\sc opt}'}_{k+1}$.
Let $\tilde{\Sigma}_i^2:=\sigma_{\tau(1)}^2+\dots+\sigma_{\tau(i)}^2$ be the variance of $\stp{\tau(1)} + \cdots + \stp{\tau(i)}$. 
From Lemma~\ref{LB} it then follows that
\begin{equation*}
\E W_{k+1}\geq  \frac12 \left(\tilde{\Sigma}_k+\tilde{\Sigma}_\ell+\sqrt{\tilde{\Sigma}_k^2-\tilde{\Sigma}_\ell^2}\right).
\end{equation*}

Recall that $\E W^{\text{\sc opt}'}_{k+1}$ was the optimal expected waiting time when $\tau(i)\leq  k$ whenever $i\leq  k$. Therefore, we have $\tilde{\Sigma}_k=\Sigma_k$ and $\sigma_k^2=\max\{\sigma_{\tau(1)}^2,\dots,\sigma_{\tau(k)}^2\}$. Now note that 
\begin{equation*}
\Sigma_k+\tilde{\Sigma}_\ell+\sqrt{\Sigma_k^2-\tilde{\Sigma}_\ell^2}
\end{equation*}
is largest when $\tilde{\Sigma}_\ell^2$ is as close to $\frac12\Sigma_k^2$ as possible. As $\sigma_k^2$ is largest of the $\sigma_{\tau(i)}^2$ with $i\leq  k$, we can always choose $\ell$ such that 
\begin{equation*}
\frac12\Sigma_{k-1}^2\leq  \tilde{\Sigma}_\ell^2\leq  \frac12\Sigma_{k-1}^2+\sigma_k^2.
\end{equation*}
This choice of $\ell$ provides us with the lower bound
\begin{equation*}
\E W^{\text{\sc opt}'}_{k+1} \geq  \frac12\frac{1}{\sqrt{2\pi}}\left(\Sigma_k + \sqrt{\frac12\Sigma_{k-1}^2} +\sqrt{\frac12\Sigma_{k-1}^2+\sigma_k^2}\right),
\end{equation*}
valid for any sequence. 
Comparing with \eqref{eq:normalub}, we obtain 
\begin{equation*}
\frac{\E \Wsvf_{k+1}}{\E W^{\text{\sc opt}'}_{k+1}} \leq 
2(\Sigma_k+\Sigma_{k-1})\bigg/\left(\Sigma_k + \sqrt{\frac12\Sigma_{k-1}^2} +\sqrt{\frac12\Sigma_{k-1}^2+\sigma_k^2}\right).
\end{equation*}
As $\Sigma_k^2=\Sigma_{k-1}^2+\sigma_k^2$, this fraction only depends on the relative size of $\Sigma_{k-1}^2$ compared to $\sigma_k^2$. Suppose that $\Sigma_{k-1}^2=c\sigma_k^2$, for some $c\geq  0$. Then $\Sigma_k^2=(c+1)\sigma_k^2$, and the fraction becomes
\begin{equation*}
\frac{2(\sqrt{c+1}+\sqrt{c})}{\sqrt{c+1}+\sqrt{\frac12c}+\sqrt{\frac12c+1}} =:f(c).
\end{equation*}
It can easily be seen that $f$ is increasing, and that $f(c)\to 4(\sqrt2-1)$ as $c\to\infty$. 

We now know that $\E \Wsvf_{k+1}/\E W^{\text{\sc opt}'}_{k+1}\leq4(\sqrt2-1)\approx 1.6569$.
By Lemma~\ref{OneWaitingTime} the same is then also true for the cost function.
This proves Theorem~\ref{thm:normal}. 
\end{proofof}


\subsection{Numerically improving the bound of Theorem \ref{nonsymmComp}}\label{lognormal}

Under the dilation ordering assumption, we have proved that $\varrho_{\omega}\leq 4$, and we also proved that $\varrho_{\omega}\leq 2$ when the service times have symmetric distributions. This suggests that one can find an upper bound on $\varrho_{\omega}$ between 2 and 4 for service-time distributions that have some degree of symmetry, but are not fully symmetric. 

Here we introduce a method to split the service time distributions into a symmetric and a nonsymmetric part, thus isolating the effect of the asymmetry on the upper bound. This can be used to numerically compute an upper bound on $\varrho_{\omega}$ (lower than 4, that is) for given problem instances. In this section we do so for lognormal service time distributions that fit real data in \cite{Cayirli}. We still impose the dilation ordering assumption.

We introduce the method for continuously distributed service times to simplify the exposition, noting that extending the method to non-continuous distributions is straightforward. Suppose $X_i$ has density $f_i(x)$. We set $g_i(x):=\min\{f_i(x),f_i(-x)\}$, $h_i(x):=f_i(x)-g_i(x)$ and $p_i:=\int_\R h_i(x)\mathrm{d}x$. Then we let $U_i$ be a random variable with density $g_i(x)/(1-p_i)$. We let $A_i$ be a random variable, independent of $U_i$, with density $h_i(x)/p_i$. Let $\bern_i$ be a Bernoulli variable taking the value one with probability $p_i$, independent of $A_i$ and $U_i$. We thus have
\begin{align*}
    X_i\dis U_i(1-\bern_i) + A_i\bern_i.
\end{align*}
Note that this construction is such that $U_i$ has a symmetric distribution around zero. One thus has that $U_i$ corresponds to the symmetric part of $X_i$, whereas $A_i$ to the nonsymmetric part. Note that $\E X_i=0$ and $\E U_i=0$, so we must have $\E A_i=0$. Let $A_i'$ have the same distribution as $A_i$, so that $A_i'$ is independent of all the other random variables. 
Since $\E[A_i'\bern_i\,|\,\bern_i] = 0$, we have
\begin{align*}
    \E\left[U_i(1-\bern_i)+(A_i-A_i')\bern_i\,\big|\,U_i(1 - \bern_i)+A_i\bern_i\right]= U_i(1-\bern_i) +A_i\bern_i.
\end{align*}
By Lemma~\ref{condExp} we conclude
\begin{align*}
    X_i\cleq U_i(1-\bern_i) + (A_i-A_i')\bern_i.
\end{align*}
As $W_{k+1}$ is a convex function in each of the $X_i$, we can then replace each $X_i$ by this upper bound in convex order to get an upper bound on $\E\Wsvf_{k+1}$. Using Lemma~\ref{reflectedUB}, we then find the following upper bound:
\begin{align}\label{compa}
    \E \Wsvf_{k+1} \leq \E|X_1+\dots+X_k|+\E|A_1\bern_1 + \dots+A_k\bern_k|.
\end{align}

Note that the second term in the right-hand side of \eqref{compa} can be numerically evaluated to any desired level of precision, e.g.\ by simulation. 
The resulting upper bound can thus be compared numerically to the lower bound $\E W^{\text{\sc opt}'}_{k+1}\geq\E(X_1+\dots+X_k)^+$, for each $k$, valid under the dilation ordering assumption. Combining the above with Lemma~\ref{OneWaitingTime}, this leads to the bound given in the next theorem. 

\begin{theorem}\label{betterThan4}
Under the dilation ordering assumption, we have
\begin{align}\label{UBA}
    \varrho_{\omega}\leq 2+2\max\bigl\{\E|A_1\bern_1 + \dots+A_k \bern_k|/\E|X_1+\dots+X_k|: k=1,\dots,n\bigr\}.
\end{align}
\end{theorem}

The more symmetric the service times and thus the random variables $X_i$, the smaller the $p_i$ and hence also the upper bound in (\ref{UBA}). When the service times are completely symmetric, the asymmetric parts $A_i$ will be zero, and we recover the upper bound of 2 of Theorem~\ref{generalComp}.

Note that the upper bound in Theorem~\ref{betterThan4} is much easier to numerically compute or simulate than $\varrho_{\omega}$ itself, as for the latter one needs to go over all $n!$ possible sequences to find the optimal one. Also, this method can be used to find an upper bound on $\varrho_{\omega}$ for any problem instance where the service times come from a finite set of distributions and an upper bound on $n$ is given, as illustrated in the next example.

\begin{table}
    \centering
        \renewcommand{\arraystretch}{1.2}
    \begin{tabular}{|l|c|c|}
        \hline
        Group & \hspace{1.1cm}Mean \hspace{1cm} & Standard deviation\\
        \Xhline{4\arrayrulewidth}
        Return & $15.50$ & $5.038$\\ \hline
        New & $19.09$ & $6.85$\\ \hline
    \end{tabular}
    \caption{Parameters of the lognormal distributions fitted by \cite{Cayirli}.}\label{tbl:cayirli}
\end{table}

\begin{example} 
We base this example on the distributions fitted to health care data in \cite{Cayirli}. There, patients were divided in two groups: new and return patients. For both groups, lognormal distributions were found as a good fit to the data used in the paper, with parameters as shown in Table~\ref{tbl:cayirli}.
We checked that problem instances coming from these two distributions satisfy both $\E[ \ln B_1]\leq \dots \leq \E[\ln B_n]$ and $\Var( \ln B_1)\leq \dots\leq \Var( \ln B_n)$, and therefore satisfy the dilation ordering assumption. It was also mentioned that the doctor that provided the data sees 10 patients per session.

We now consider 11 problem instances, each of them corresponding to  $n=10$ patients. In the $k$-th instance $k-1$ patients have the first lognormal service-time distribution, whereas the remaining $11-k$ patients have the other lognormal distribution, with $k=1,\ldots,11$.
We compute the upper bound in (\ref{UBA}) for each of these instances through simulation. Note that we can simulate $A_iJ_i$ by first drawing a random instance of $X_i$, and then setting $A_i$ equal to $X_i$ with probability $h_i(X_i)/f_i(X_i)$ and equal to zero otherwise. Simulating 1\,000\,000 random instances of $X_1,\dots,X_n$ then allows us to compute $\E|X_1+\dots+X_k|$ and $\E|A_1J_1+\dots+A_kJ_k|$ for $k=1,\dots,n$, and thus the upper bound in (\ref{UBA}) for each of the 11 problem instances. Note that smaller problem instances are automatically included because we compute $\E|X_1+\dots+X_k|/\E|A_1J_1+\dots+A_kJ_k|$ for all $k$.
We find that $\varrho_{\omega}\leq 3.43$ for any problem instance consisting of at most 10 patients, with service times that follow one of the two lognormal distributions. We thus conclude that in this example $\varrho_{\omega} \leq 3.43$ for any problem instance.
\end{example}


\section{Asymptotic optimality of \svf{} under mean-based schedules}\label{sec:asymptotic}

The problem of finding the optimal sequence gets increasingly difficult as the number of patients grows large. In the numerical experiments described in Appendix~\ref{MBNumerics} we could only find the optimal sequence for up to 10 or 11 patients, but in some applications more than 20 patients need to be scheduled in a session (see e.g.~\cite{Klassen}). Therefore, we would like to know how the approximation ratio behaves as $n$ grows larger. 
In this section we assess the performance of the \svf{} sequence as the number of patients tends to infinity.
Throughout this section we  assume that the schedule is mean-based: the time planned for each appointment is equal to the corresponding mean service time. 
The goal in this section is to prove that the \svf{} sequence is asymptotically optimal as the number of patients tends to infinity, under a weak assumption. 

We consider the setting in which we are given, for each value of $n$, a vector ${\bm{B}}_n = (\sti{1}{n}, \sti{2}{n}, \ldots, \sti{n}{n})$ of service-time distributions.
For $i \leq n$, let $\mean{i}{n}$ and $\sdev{i}{n}^2$ denote the mean and variance of $\sti{i}{n}$, and let $\stpn{i} := \sti{i}{n} - \mean{i}{n}$ for all $i \leq n$. 
Similarly, 
$W_{n,i}$, and $I_{n,i}$ are all with respect to the service-time distributions ${\bm{B}}_n$, and an implicit fixed permutation $\tau \in \sym_n$.
Note that Example~\ref{noconstbound} shows that $\varrho_\omega$ cannot be bounded without imposing some assumption, even when $n$ tends to infinity. 
We do not need Assumption~\ref{AStochDom} (dilation ordering assumption) in this section.
Instead, we impose the following assumption, similar to the Lyapunov condition of the Lyapunov-version of the central limit theorem (CLT), concerning the order of magnitude of the $(2+\delta)$-th moments of the service times.
The difference between our assumption and the conventional Lyapunov condition is the supremum  over all $n\geq k$ and all sequences $\tau$.
\begin{assumption}\label{assumption}
We assume that there exists a $\delta>0$ such that, as $k\to\infty$,
\begin{align*}
    q_k:= \sup_{n\geq  k,\tau \in \sym_n} \frac{1}{\sqrt{\sum_{i=1}^k \sdev{\tau(i)}{n}^2}^{2+\delta}} \sum_{i=1}^k \E|\stpn{\tau(i)}|^{2+\delta}\to 0.
\end{align*}
\end{assumption}

The main result of this section is the following.
\begin{theorem}\label{!}
    Under Assumption $\ref{assumption}$, $\varrho_{\omega}({\bm{B}}_n) \to 1$ as $n\to\infty$.
\end{theorem}

To obtain some intuition about why such a result is reasonable, 
consider normally distributed service times.
Fix some $n$ and $k \leq n$ (both large), and consider the waiting time $W_{n,k+1}$ corresponding to some fixed permutation $\tau$.
Let $(S_j)_{j \leq k}$ be the random walk for which $W_{n,k+1}$ is the maximum; so $S_j = \sum_{i =1}^j X_{n,\tau(k+1-i)}$.
Because the steps are normally distributed, we can embed this random walk within a standard Brownian motion $V(t)$. 
More precisely, we can define the coupling $S_{j} = V(\theta_{j})$, where $\theta_{j} := \sum_{i =1}^j \sigma_{n,\tau(k+1-i)}^2$.
Assumption~\ref{assumption} then tells us simply that no individual service time has a non-negligible fraction of the total variance, in the large $n$ limit.
Thus the gaps $\theta_{j+1} - \theta_{j}$ between sample times become negligibly small in comparison to the interval of interest, namely $[0, \theta_{k}]$.
So the maximum $W_{n,k+1}$ of the random walk is very well approximated by the supremum of the Brownian motion on the interval $[0, \theta_{k}]$.

This supremum is of course completely understood.
By the reflection principle, we know that $\sup_{0 \leq t \leq \theta_k} V(t)$ has the distribution of the absolute value of a normal with variance $\theta_k$ and mean zero.
Since the \svf{} ordering minimizes $\theta_k$ (simultaneously for all choices of $k$), it is thus optimal in the limit.

\medskip

Extending this intuition to general distributions and making it rigorous requires care.
The plan is to apply the reflection principle to the random walk, so that hopefully the magnitude of the final position of the walk
is a good estimate of the maximum of the unreflected walk.
More precisely, fix some value $a > 0$, and let $T(a)$ denote the first step at which the random walk $(S_j)_{j \leq k}$ reaches $a$ (or $T(a) = \infty$ if this event does not occur).
Now define
\begin{equation*}
\hat{S}_{j}(a)=\begin{cases}
    S_{j} &\text{ if } j<T(a) \\
    2a-S_{j} &\text{ if } j \geq  T(a)
\end{cases}
\qquad \text{and} \qquad
\tilde{S}_{j}(a)=\begin{cases}
    S_{j} &\text{ if } j <T(a) \\
    2S_{T(a)} - S_{j} &\text{ if }j\geq  T(a);
\end{cases}
\end{equation*}
cf.~\eqref{reflected}.
So $\hat{S}_j(a)$ is the walk reflected immediately as it ``crosses'' $a$ for the first time, and $\tilde{S}_j$ the walk reflected from the first step after crossing $a$ (see Figure~\ref{three}).

\begin{figure}[h]
\centering
\includegraphics[scale=0.6]{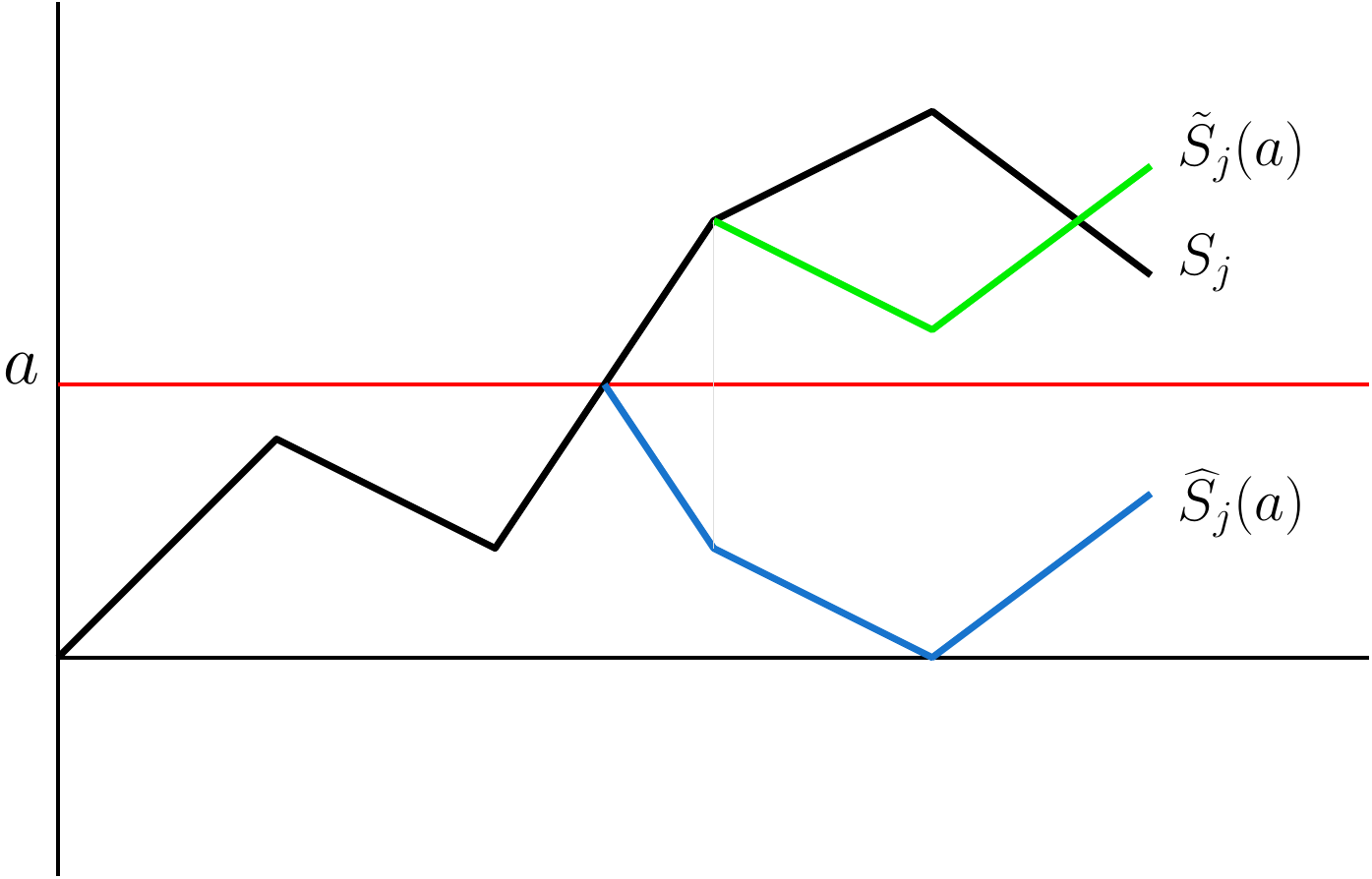}
\caption{The processes $S_{j}$, $\hat{S}_{j}(a)$ and $\tilde{S}_{j}(a)$. The upper horizontal line indicates level $a$.
}\label{three}
\end{figure}

The event that $W_{n,k+1}$ is at least $a$ occurs precisely when $T(a)$ is finite, which implies that either $S_{k} \geq a$ or $\hat{S}_k(a) = 2a-S_{k} > a$. 
As these are disjoint events, we have
\[ 
    \p(W_{n,k+1}\geq  a)=\p(S_{k}\geq  a)+\p(\hat{S}_{k}(a)>a);
\]
this is along similar lines to the proof of Lemma~\ref{reflectedUB}.
Note that the processes $\hat{S}_{j}(a)$ and $\tilde{S}_{j}(a)$ have the same increments, except for step $T(a)$. 
In this step the increments differ by $2(S_{T(a)}-a)$, twice the amount by which the random walk ``overshoots''  level $a$. 
As this overshoot is nonnegative and bounded by $\max_{1\leq  i\leq  k}\stpn{\tau(i)}$, we find that
\begin{equation}\label{hatS}
    \tilde{S}_{k}(a)\geq  \hat{S}_{k}(a)\geq  \tilde{S}_{k}(a)-2\max_{1\leq  i\leq  k}\stpn{\tau(i)}.
\end{equation}
This leads to  the estimates
\begin{align}\label{eq:walkbounds}
\p\left(S_{k}\geq  a\right)+\p\left(\tilde{S}_{k}(a)>a\right)&\geq  \p\left(W_{n,k+1}\geq  a\right) \nonumber\\
                                       &\geq  \p\left(S_{k}> a\right)+ \p\left(\tilde{S}_{k}(a)>a+2\max_{1\leq  i\leq  k}\stpn{\tau(i)}\right).
\end{align}

We now come to the first difficulty: we would like to estimate the probabilities $\p(S_k > a)$ and $\p(\tilde{S}_k(a) > a)$ by applying a CLT result.
However, the reflected process $(\tilde{S}_j)$ does not have independent steps (unless the service time distributions are symmetric around their mean, which we do not assume).
Fortunately, $(\tilde{S}_j)$ \emph{is} 
a martingale, and so we can apply the following CLT-type result for martingales.
Here and in the remainder of this section, $Z$ denotes a standard normal random variable.

\begin{theorem}[Heyde and Brown~\cite{Heyde}]\label{CLT}
Let $(\xi_i,\mathcal{F}_i)$ be a sequence of martingale differences, and let $Y_j=\xi_1+\dots+\xi_j$ be the corresponding martingale. Suppose that the conditional variance, given by
\[ 
    \sum_{i=1}^k \E[\xi_i^2\,|\,\mathcal{F}_{i-1}],
\]
is equal to one for some $k$. Then for any $\delta>0$ there exists a constant $C_\delta$ that depends on $\delta$ only, such that 
\[
    \sup_{x\in\R}|\p(Y_k> x)-\p(Z> x)|\leq  C_\delta\left(\sum_{i=1}^k \E|\xi_i|^{2+\delta}\right)^{1/(3+\delta)}.
\]
\end{theorem}

This is exactly what we need to obtain the following proposition (the full derivation can be found in Appendix~\ref{App5}).
\begin{proposition}\label{UB2}
For any $k$, $n \geq k$ and permutation $\tau$, we have
\begin{align*}
    \frac{\E W_{n,k+1}}{\sqrt{\sum_{j=1}^k \sigma^2_{n,\tau(j)}}} \leq  \E|Z| +2(C_\delta+1 )\sqrt{q_k}^{1/(3+\delta)},
\end{align*}
and $C_\delta$ is as in Theorem~\ref{CLT}. 
\end{proposition}

Next, we come to the lower bound.
Considering \eqref{eq:walkbounds}, the second difficulty becomes apparent: 
we do not make any boundedness assumptions on the service time distributions, and so $\max_{1 \leq i \leq k} X_{n,\tau(i)}$ could be very large.
This is remedied with a truncation argument.

We will consider the random walk $(S'_j)_{j \leq k}$ with steps $\stpn{\tau(i)}\mathbbm{1}_{\stpn{\tau(i)}\leq  c_{n,k}}$ instead of $\stpn{\tau(i)}$ (where $\mathbbm{1}_E$ denotes the indicator of an event $E$).
Here, $c_{n,k}$ is a bound depending on $n$ and $k$, but not $i$.
Let $W_{n,k+1}'$ be the maximum of the new random walk; 
clearly $\E W_{n,k+1}\geq  \E W_{n,k+1}'$.
Further, defining $\hat{S}'_j(a)$ and $\tilde{S}'_j(a)$ analogously to $\hat{S}_j(a)$ and $\tilde{S}_j(a)$ but with respect to the walk $(S'_j)$, \eqref{hatS} now yields
\[
    \hat{S}_{k}(a)\geq  \tilde{S}_{k}(a)-2c_{n,k}.
\]
This allows us to apply Theorem~\ref{CLT} to obtain a lower bound on $\E W'_{n,k+1}$ and hence $\E W_{n,k+1}$.
(A technical complication is that $S'_j$ and $\tilde{S}'_j(a)$ are no longer quite martingales; however, this can be overcome.)
The value $c_{n,k}$ must be chosen carefully 
in a way that balances the need to sufficiently bound the overshoot (thus making \eqref{eq:walkbounds} effective) and the need to not affect the steps too much (since this would decrease the maximum significantly). 
The result is the following proposition; further details can be found in Appendix~\ref{App5}.

\begin{proposition}\label{LB2}
Under  Assumption~\ref{assumption}, for each $\varepsilon>0$ there exists a $K$ depending  on $\varepsilon$ only, such that for all $k\geq  K$, $n\geq  k$ and permutations $\tau$,
\[
    \frac{\E W_{n,k+1}}{\sqrt{\sum_{i=1}^k \sigma^2_{n,\tau(i)}}}\geq  (1-\varepsilon)\E|Z|.
\]
\end{proposition}
Combining Proposition~\ref{UB2} and Proposition~\ref{LB2} yields Theorem~\ref{!} in a straightforward manner.
For completeness, the details can be found in Appendix~\ref{App5}.

\begin{remark}
As in Remark \ref{REM1}, when the scheduled session end time is equal to the expected total service time, the expected overtime $\E W_{n,n+1}$ can be handled similarly to waiting time, and the result of Theorem~\ref{!} is also valid when some extra term $c\,\E W_{n,n+1}$ (with $c>0$) is added to the cost function.\end{remark}

\begin{remark} In this remark we assess the rate at which $\varrho_{\omega}({\bm B}_n)$ converges to 1.
Suppose that $\inf_{n,i} \{\sigma_{n,i}^2\}>0$ and $\sup_{n,i}\{\E|X_{n,i}|^{2+\delta}\}<\infty$. Then $q_k=O(k^{-\delta/2})$ indeed converges to zero.
Following the steps of the proof of Theorem~\ref{!} we then find
\begin{align*}
    \varrho_{\omega}({\bm B}_n)= 1+O\left(\frac{K}{n}\right)+O\left(\sqrt{q_K}^{1/(3+\delta)}\right)=1+O\left(\frac{K}{n}\right)+O\left(K^{-\delta/(12+4\delta)}\right).
\end{align*}
To obtain some insight into the convergence rate, observe that by choosing $K$ in a way that these terms are balanced, it follows that
\begin{align*}
    \varrho_{\omega}({\bm B}_n)=1+O\left(n^{-\delta/(12+5\delta)}\right).
\end{align*}
Note that for many practical distributions, including lognormal distributions, all moments exist. In such a case, Theorem~\ref{!} can be applied with any choice of $\delta$.
\end{remark}


\section{Bounds on performance for optimally-spaced schedules}\label{sec:combined}

The previous sections focused on mean-based schedules, i.e.\ schedules where the interarrival times are equal to the mean service times. Rather than mean-based schedules, one would preferably use optimally-spaced schedules. The problem of finding optimally-spaced schedules is well understood (see e.g.~\cite{Begen, Kuiper}). In this section we consider the performance of the \svf{} sequence compared to the optimal combination of sequence and schedule.
For this case we have also performed extensive numerical experiments to gain insight into what performance can be expected for \svf{} for practical distributions, described in more detail in Appendix~\ref{OptNumerics}. As in the  mean-based case, we focused on exponential and lognormal service times. 
Our main finding was that for all experiments we performed, the \svf{} sequence was the optimal sequence for all experiments with these distributions. 
It is noted, though, that such experiments can only be done for relatively small numbers of $n$, due to the complexity involved in computing the optimal sequence and schedule. 

Kong et al.~\cite{Kong} give an example demonstrating that \svf{} is not always optimal for optimally-spaced schedules, but their example does not satisfy the dilation ordering assumption (Assumption~\ref{AStochDom}). 
    The following example shows that even when Assumption~\ref{AStochDom} applies, \svf{} may not be optimal.

\begin{example}\label{OptimalCounterExample}
Suppose we have $n=7$ patients, and we set $\omega=\frac12$. The first three patients have service times that take value 0 or 2, each with probability $\frac12$. The other four patients have service times that take values 0 or 4 each with probability $\frac14$ and that take value 2 with probability $\frac12$. As $B_i-2$, for $i=4,5,6,7$, has the same distribution as $B_1-1+A$, where $A$ takes values $-1$ or 1 each with probability $\frac12$ independent of $B_1$, we can see through Lemma~\ref{condExp} that $B_1\leq_{\text{dil}} B_i$. So this example satisfies the dilation ordering assumption.

We computed the optimal interarrival times and corresponding cost for each possible sequences (per the method described in Appendix~\ref{OptNumerics}). The ratio between the optimal cost for the \svf{} sequence and the optimal cost overall was found to be $1.0787$, so the \svf{} sequence is not optimal.
\end{example}

The goal of the remainder of this section is to prove upper bounds on the approximation ratio $r_{\omega}$. In Section~\ref{main4}, we will do so when the service-time distributions are from the same location-scale family, leading to Theorem~\ref{approx}. 
In Section~\ref{examples}, we discuss the implications of this theorem for specific location-scale families of interest.
Then in Section~\ref{lognormal2}, we move beyond location-scale families and consider the lognormally distributed service times, which are of particular interest in practice \cite{Cayirli,Klassen}.
Finally, in Section \ref{AsymptoticExample} we show that for optimally-spaced schedules \svf{} is \emph{not} asymptotically optimal, in stark contrast to the situation for mean-based schedules.


\subsection{Location-scale family of service times}\label{main4}

We impose the following assumption.

\begin{assumption}\label{family}
    The distributions $(B_i)_{1 \leq i \leq n}$ form a location-scale family.
    In other words, there exists a random variable $B$ having mean zero and variance one such that $B_i\dis \mu_i+\sigma_iB$.
\end{assumption}

Note that, by \cite{Shaked}, Theorem 3.A.18, this assumption implies Assumption \ref{AStochDom}. Due to Example~\ref{noconstbound}, no bound on the approximation ratio $r_\omega$ can be found without imposing \emph{some} assumption on the service-time distributions.

As well as using \svf{} as a sequencing rule, we must now specify an (ideally simple) scheduling rule as well.
We use a schedule of the form $\bm{x}=\bm{\mu}+\alpha\bm{\sigma}$ for some $\alpha>0$, in line with a suggestion by \cite{Charnetski}.
So we include additional slack in the schedule after each patient proportional to its standard deviation, in order to prevent the propagation of delays.
We will choose 
\begin{align}\label{alphaatje}
\alpha=\sqrt{\frac{1-\omega}{2\omega}+\frac{\sigma_{n-1}}{2\sum_{i=1}^{n-1}\sigma_i}};
\end{align}
the rationale behind this choice will become clear from the proof, where it will turn out to optimize the bound we obtain.
Let $Q_B$ denote the quantile function of $B$, i.e.\ $Q_B(y)=\inf\{x: y\leq \p(B\leq x)\}$. Define $B(\omega)= B-Q_B(1-\omega)$, and 
\begin{align*}
K(B,\omega)= \sqrt{2\omega}/\left[\omega \E B(\omega)^-+(1-\omega)\E B(\omega)^+\right].
\end{align*}
The main result of this section is the following.

\begin{theorem}\label{approx}
    Suppose that, for the \svf{} sequence, we use the schedule $\bm{x}=\bm{\mu}+\alpha\bm{\sigma}$, with $\alpha$ given by~$(\ref{alphaatje})$. Under Assumption $\ref{family}$, we have $r_{\omega} \leq K(B,\omega)$.
\end{theorem}

This result follows immediately from the bounds on the cost function $\Obj{\tau}{\bm{x}}{\omega}$ given in the following two propositions, that are proved in Appendix~\ref{App4}.

\begin{proposition}\label{SVF}
Suppose $\alpha$ is given by~$(\ref{alphaatje})$.
Under Assumption $\ref{family}$,
\begin{align*}
\Obj{\Id}{\bm{\mu}+\alpha\bm{\sigma}}{\omega}\leq  \sqrt{2\omega}\sum_{i=1}^{n-1} \sigma_i.
\end{align*}
\end{proposition}

\begin{proposition}\label{PropFreeLB}
Under Assumption $\ref{family}$, for any sequence and schedule,
\begin{equation*}
\Obj{\tau}{\bm{x}}{\omega}
 \geq  \left[\omega\E B(\omega)^- +(1-\omega)\E B(\omega)^+ \right]\sum_{i=1}^{n-1} \sigma_i.
\end{equation*}
\end{proposition}

The idea behind the proof of Proposition~\ref{SVF} is as follows. We use that the waiting time can be expressed as the maximum of a random walk, as per \eqref{W}. An upper bound for this maximum can now be found by using a comparison with another random walk that has i.i.d.\ steps, each distributed as the step of the original random walk with the largest variance. This upper bound is found by noting that if one (i) splits the steps in two parts, and (ii) multiplies the last part by some constant larger than one (leaving the first part unchanged), then the maximum increases. For the maximum of the new i.i.d.\ random walk, the classical \emph{Kingman's bound} can be applied. After thus finding an upper bound on the expected waiting time, the expected idle time can then also be bounded using \eqref{I}.

The idea behind the proof of Proposition~\ref{PropFreeLB} is to write
\begin{align*}
\omega\E I_{k+1}+(1-\omega)\E W_{k+1}=\omega\E(W_k+B_{\tau(k)}-x_{\tau(k)})^-+(1-\omega)\E(W_k+B_{\tau(k)}-x_{\tau(k)})^+,
\end{align*}
and minimize this over $W_k-x_{\tau(k)}$. This minimization problem is the classical \emph{newsvendor problem}, which has a known solution. This results in a lower bound on the cost function that is independent of the schedule. This lower bound can also be easily minimized over the sequences, resulting in Proposition~\ref{PropFreeLB}.

When $\omega=\frac12$, i.e.\ when waiting time and idle time are equally important, we know that $Q_B(\frac12)$ is equal to the median of $B$. In this case, we have $K\left(B,\frac12\right)=2/\E|B-m|$, where $m$ is the median of $B$.


\subsection{Examples}\label{examples}

In this subsection we present  examples of location-scale families for which we can compute $K(B,\omega)$ or $K\left(B,\frac12\right)$ from Theorem~\ref{approx}, so as to obtain insight into the magnitude of the constant $K(B,\omega)$. For the location-scale families of normal, uniform, shifted exponential and Laplace distributions, the results are shown in Table \ref{table}. For normal distributions $K(B,\omega)$ is not shown, as the expression does not simplify (with respect to the one presented  in Theorem~\ref{approx}).

{\small
\begin{table}[h]
\centering
\bgroup
\def\arraystretch{1.4}
\begin{tabular}{|l|c|c|}
\hline
Location-scale family & $K(B,\omega)$ & $K\left(B,\frac12\right)$ \\ \Xhline{4\arrayrulewidth}
Normal & See Theorem~\ref{approx} & $\sqrt{2\pi}\approx 2.51$ \\ \hline
Uniform & $\displaystyle\frac{1}{1-\omega}\sqrt{\frac{2}{3\omega}}$ & $\frac43\sqrt3\approx 2.31$ \\ \hline
Shifted exponential & $\displaystyle-\frac{\sqrt2}{\sqrt{\omega}\ln(\omega)}$ & $ {2}/{\ln(2)}\approx 2.89$ \\ \hline
Laplace & $\displaystyle\frac{2\sqrt{\omega}}{\min\{\omega,1-\omega\}(1-\ln(2\min\{\omega,1-\omega\}))}$ & $2\sqrt2\approx 2.83$ \\
\hline
\end{tabular}
\egroup
\caption{The values of $K(B,\omega)$ and $K\left(B,\frac12\right)$ for some location-scale families.} \label{table}
\end{table}}

Now  consider the case of Pareto (of type II, that is) distributions. A random variable $X$ has such a distribution if 
\begin{align*}
\p(X>x) = \left(1+\frac{x-\mu}{\sigma}\right)^{-\beta} \text{ for } x\geq  \mu,
\end{align*}
for certain parameters $\mu, \sigma>0, \beta$.
The Pareto  distributions with fixed parameter $\beta$ form a location-scale family. Suppose that the $B_i$ have Pareto distributions with fixed parameter $\beta>2$. Then
\begin{align*}
K(B,\omega) =\sqrt{\frac{2\omega\beta}{\beta-2}}\bigg/\left[-2\omega^{-\beta(\beta-1)+1} +\omega^{-\beta(\beta-1)}-(\beta-1)\omega^{\beta+1} +\beta\omega\right].
\end{align*}
In addition,
\begin{align*}
K\left(B,\tfrac12\right)=2\sqrt{\frac{\beta}{\beta-2}}\bigg/ \left[\beta-\left(\tfrac12\right)^\beta(\beta-1)\right].
\end{align*}

For most typical location-scale families, the value of $K\left(B,\frac12\right)$ is between 2 and 3. However, in the Pareto case the value becomes much larger when $\beta$ approaches two. Also, for $\omega$ close to either one of the extremes 0 or 1, the constant $K(B,\omega)$ blows up.


\subsection{Lognormally distributed service times}\label{lognormal2}

In this subsection we focus on the case of lognormal service times. The scheduling rule that we use here is such that the interarrival times are a multiple (larger than 1) of the mean service times, rather than the mean service times increased by  a multiple of the corresponding standard deviations.
    As will become clear below, this type of scheduling rule turns out to be convenient for the case of lognormal distributions. Note that, with $m_i:=\E[\ln B_i]$ and $s_i^2:=\Var(\ln B_i)$, we have that $\E B_i = \exp(m_i + s_i^2/2)$, so that the rule involves both parameters of the lognormal distribution.
We have the following result.

\begin{theorem}\label{rLognormal}
Suppose the $B_i$ are lognormally distributed with $m_1\leq \dots\leq m_n$ and $s_1^2\leq \dots\leq s_n^2$. When we use the schedule $\bm{x}=(1+\alpha)\bm{\mu}$ for the \svf{} sequence, with
\begin{align*}
\alpha=\frac{1}{\sqrt{2\omega}} \sqrt{(\exp(s_{n-1}^2)-1)},
\end{align*}
then
\begin{align*}
    r_{\omega}\leq 2\omega\alpha \cdot \bigl[(1-\omega)\p(Z\geq Q_Z(1-\omega)-s_1)-\omega\,\p(Z\leq Q_Z(1-\omega)-s_1)\bigr]^{-1}.
\end{align*}
\end{theorem}

The proof can be found in Appendix~\ref{App4}. 
It is based on a similar idea as the one used in the proof of Theorem~\ref{approx}. The main difference is in the upper bound, where it needs to be proved that the i.i.d.\ random walk used for comparison indeed has a larger expected maximum. For lognormal distributions, we use a convex ordering among the stepsize distributions to prove this, noting that the maximum of a random walk is a convex function in the stepsizes. 

As an example, we apply Theorem~\ref{rLognormal} to the data found in \cite{Cayirli}. 
Recall from Section \ref{lognormal} that the patients were divided into ``new'' and ``return'' patients, with service times fitted by lognormal distributions with parameters given in Table~\ref{tbl:cayirli}.
It can be checked that any problem instance containing a mix of these patient groups satisfies the assumptions of Theorem~\ref{rLognormal}. For any such problem instance, the largest possible $s_{n-1}$ corresponds to a new patient, and the smallest possible $s_1$ corresponds to a return patient. When setting $\omega=\frac12$, calculating the upper bound in Theorem~\ref{rLognormal}, we find for this example that $r_{\frac{1}{2}}\leq 2.90$.


\subsection{Example where \svf{} is not asymptotically optimal}\label{AsymptoticExample}

In Section \ref{sec:asymptotic} we showed that the \svf{} sequence is asymptotically optimal as $n\to\infty$ when mean-based schedules are used. 
Perhaps surprisingly, this is no longer true when optimally-spaced schedules are used.
The example described in this section satisfies the condition used in the mean-based case, Assumption~\ref{assumption}.
Note that the example does not satisfy Assumption~\ref{AStochDom} {(the dilation ordering assumption)}. 

The example is as follows. We set $\omega=\frac12$, and we have two groups of patients, each consisting of $n/2$ patients (where $n$ will be large).
Let $c$ be some positive integer that we will eventually allow to grow very large, and let $a$ be a fixed constant larger than $1$ that we will choose later.
The service time distributions of the two groups are given by
\begin{equation*}
    B^{(1)} =\begin{cases} c&\text{ with prob. }p:=\frac{1}{1+c^2}, \\
-\frac{1}{c} &\text{ with prob. } 1-p,
\end{cases} \qquad \qquad
B^{(2)} =\begin{cases} a&\text{ with prob. }\frac12,\\
-a &\text{ with prob. }\frac12.
\end{cases}
\end{equation*}
These are chosen to have mean zero for convenience; by shifting them (which would shift the optimal schedule by the same amount), these can be adjusted to have nonnegative support.

The variance of the service times of the group 1 (group 2) patients  is $1$ ($a^2 > 1$, respectively).
This means that the \svf{} sequence first serves all patients of group 1, followed by all patients from group 2.
We will compare the \svf{} sequence with the \emph{mixed sequence},
where patients of group 1 and 2 are served alternatingly starting with group 1. 
As we will see, for large enough $c$ the mixed sequence will be cheaper under optimal scheduling, thus refuting the asymptotic optimality of \svf{}.
Since we will consider a {large-$n$} limit, we will consider the average per-patient cost, consisting of the average expected waiting time and idle time per patient.

We begin by analyzing the \svf{} sequence.
As we need a lower bound on its cost, we need to determine the optimal schedule.
Fortunately, in the limit of large $n$, the situation simplifies.
Each group can be treated as a stationary queue, with an associated expected waiting and idle time for each group at stationarity
(patients at the beginning of the group, before stationarity is achieved, make a negligible contribution to the total cost).
We use $W^{(i)}$ and $I^{(i)}$ to denote the distributions of waiting time and idle time for patients in group~$i$.
We need to determine the optimal stationary interarrival times $x^{(1)}$ and $x^{(2)}$ for patients in groups 1 and 2, respectively.
As a first remark, the average expected idle time of a patient in group $i$ is simply $x^{(i)}$; this is a consequence of the Lindley  {recursions} \eqref{LindleyModel} combined with stationarity, which yields
\[
    \E \left[W^{(i)} - I^{(i)}\right] = \E \left[ W^{(i)} + B^{(i)} - x^{(i)}\right] = \E W^{(i)} - x^{(i)}.
\]
The average expected waiting time is more challenging,
but we can compute it exactly in the limit as $c \to \infty$; we defer the argument to Appendix~\ref{AsExampleApp}.

\begin{lemma}\label{lem:svfinterarrival}
    For the \svf{} sequence in the limit $c\to\infty$, the optimal interarrival times for groups $1$ and $2$ are $x^{(1)}=\frac12\sqrt2$ and $x^{(2)}=a$ respectively.
    Patients in group $1$ incur an average expected waiting time of $\E W^{(1)} = \tfrac12\sqrt{2}$, 
    whereas patients in group $2$ incur no waiting time.
    This yields an average total cost of $\tfrac12(a+\sqrt2)$.
\end{lemma}

We now move to the mixed sequence.
Here, we need an upper bound on the cost, and so we can use any {scheduling} rule we like.
We will simply use the same interrarival times as were used in the \svf{} sequence: 
a patient in group $i$ has an interarrival time of $x^{(i)}$.
Just as with the \svf{} sequence, in the limit as $n \to \infty$ the average expected idle time is $\tfrac12(x^{(1)} + x^{(2)}) = \tfrac12(a + \tfrac12\sqrt{2})$.
The main challenge is to bound the average expected waiting time.

Let $V^{(i)}$ be the waiting time experienced by a group $i$ patient at stationarity.
Then $V^{(i)}$ is the maximum of the random walk with alternating steps distributed as $B^{(1)}-x^{(1)}$ and $B^{(2)}-x^{(2)}$.
The random walk for $V^{(1)}$ starts with a step $B^{(2)}-x^{(2)}$ and the random walk for $V^{(2)}$ starts with a step $B^{(1)}-x^{(1)}$.
This difference vanishes in the limit  {as $c$ grows large} (see Appendix~\ref{AsExampleApp} for the proof).

\begin{lemma}\label{lem:bothsame}
$\lim_{c\to\infty} \E V^{(1)}=\lim_{c\to\infty} \E V^{(2)}$.
\end{lemma}

So we focus on $\lim_{c \to \infty} \E V^{(1)}$.
Note that the group 2 steps, distributed as $B^{(2)}-a$, are never positive, so the maximum of the alternating walk can only be attained after a group 1 step. 
Therefore, we can combine the two alternating steps into one step, and $V^{(1)}$ is the maximum of a random walk $(S^{(1)}_j)_{j \geq 0}$ with steps
$B^{(1)}_j+B^{(2)}_j-(a+\frac12\sqrt2)$. 
(Here, $B^{(i)}_j$ is an independent copy of $B^{(i)}$, for each  {$i\in\{1,2\}$ and $j\in \N$}.)
Observe that $\E V^{(1)}$ is never bigger than $\E W^{(1)}$, the expected maximum of a random walk with steps distributed as $B^{(1)}-\frac12\sqrt2$, as $B^{(2)}-a\leq 0$, so as a preliminary bound we have $\E V^{(1)}\leq \E W^{(1)} = \frac12\sqrt2$.
Recalling that in the \svf{} sequence the group 2 patients experience no waiting time,
we see that our goal is to show that in fact $\E V^{(1)} < \tfrac12 \E W^{(1)} = \tfrac14\sqrt{2}$.

The simplest way to bound $\E V^{(1)}$ is via a bound by \cite{Kingman}, which says that if $(S_i)_{i \geq 0}$ is a random walk with i.i.d.\ steps distributed according to $Z$, where $\E Z < 0$, then
\[
    \E \max_{i \geq 0} S_i \leq -\frac{\Var \; Z}{2\E Z}.
\]
Applied to the random walk $S^{(1)}$, 
\[
    \E V^{(1)} \leq \frac{\Var[B^{(1)} + B^{(2)}]}{2(a + \tfrac12\sqrt{2})} = \frac{1+a^2}{2a + \sqrt{2}}.
\]
This is always larger than $\tfrac12\sqrt{2}$, for any $a \geq 1$, so this  bound {does not} suffice.
Instead,  {let us} consider the random walk $(\tilde S_j)_{j \geq 0}$ with steps $B^{(1)}_i-(a+\frac12\sqrt2)$; let $\tilde{V}$ be its maximum.
Kingman's bound for this random walk yields
\[
    \E \tilde{V} \leq \frac{\Var\; B^{(1)}}{2(a + \tfrac12\sqrt{2})} = \frac{1}{2a + \sqrt{2}}.
\]
As long as we choose $a > \tfrac32\sqrt{2}$, this is below $\tfrac12 \E W^{(1)}$.
All that remains is to show the following lemma.
\begin{lemma}\label{lem:simplerrw}
    $\lim_{c \to \infty} \E V^{(1)} \leq \lim_{c \to \infty} \E \tilde{V}$.
\end{lemma}

We sketch the proof here, and relegate the details to  {Appendix~\ref{AsExampleApp}}.
Consider a typical realization of $(\tilde{S}_j)_{j \geq 0}$.
It consists primarily of downward steps of size $a+ \tfrac12\sqrt{2} + 1/c \approx a + \tfrac12\sqrt{2}$, with occasional large upwards jumps of size $c$.
These upwards jumps, occurring with probability $p=1/(1+c^2)\approx 1/c^2$ are very rare.
One can prove that with overwhelming probability, the maximum of $S^{(1)}$ occurs within the first $2c$ steps;
condition on this for the remainder.
If there are no upwards jumps in the first $2c$ steps, then $\tilde{V}$ and $V^{(1)}$ are both zero; this happens with probability $(1-p)^{2c} \approx 1 - 2pc \approx 1 - 2/c$.
Independently of whether there are any such upwards jumps, the difference between the random walks 
for $\tilde{V}$ and for $V^{(1)}$ after $j$ steps is $\sum_{i=1}^j B^{(2)}_i$.
By an appropriate concentration bound, we can deduce that $S^{(1)}_j \leq \tilde{S}_j + c^{2/3}$ for all $j \leq 2c$, with extremely high probability.
So with probability about $1-2/c$, $V^{(1)} = \tilde{V} = 0$, and with probability about $2/c$, $V^{(1)} \leq \tilde{V} + c^{2/3}$.
This implies that $\E V^{(1)} \leq \E \tilde{V} + c^{-1/3}$, yielding the lemma.

\medskip

With the optimized choice $a=\tfrac12\sqrt{2}+\sqrt{3}$, a quantitative lower bound of 
$r_{1/2}\geq \frac12+\frac14\sqrt6 > 1.11$ can be obtained.


\section{Discussion and directions for further research}

We have shown that under quite general conditions, the \svf{} sequence yields a constant-factor approximation.
Furthermore, we have seen that additional information about the instance, such as knowing that the service-time distributions fall within a certain class, or that the number of patients is large, can lead to substantial improvements of our worst-case bounds.

For mean-based schedules, Example~\ref{BadExampleMB} and Theorem~\ref{nonsymmComp} show that the worst-case approximation ratio lies between 1.52 and 4; for symmetric service-time distributions we found the upper bound 2.
It would be interesting to reduce the gap between the lower and upper bound; we suspect neither bound is tight. 
In particular, the upper bound on the cost of the \svf{} sequence appears to be a strong bound only in the regime of many patients, with service times of similar variances; the lower bound on the cost of arbitrary sequences, on the other hand, appear strong in situations where only a few service times with large variance have a significant impact on the cost function. 
This suggests that more refined arguments, possibly considering multiple regimes, could lead to an improved upper bound.
Improving our bounds for special cases (such as normal and lognormal distributions), or considering other practically relevant service-time distributions, would also be of interest.

When optimizing over both the sequence and the schedule, we obtained bounds for location-scale families and for lognormally-distributed service times.
These bounds obtained are not uniform: in the former case, the bounds depend on $\omega$ and the location-scale family, and in the latter case, on the parameters of the lognormal distributions.
A constant-factor approximation that does not depend on these quantities, or that holds in greater generality (e.g., to all distributions satisfying the dilation ordering assumption), remains an open question.
The \svf{} sequencing rule remains a promising candidate, but a more sophisticated choice of scheduling rule will certainly be needed.

We found that the \svf{} sequence is asymptotically optimal as the number of patients tends to infinity under a mild Lyapunov-type assumption, but that this is no longer the case when optimally-spaced schedules are used, as shown in Section~\ref{AsymptoticExample}. Such an asymptotic optimality for optimally-spaced schedules might still hold under more strict assumptions on the service times. Finding what these assumptions should be and proving asymptotic optimality under these assumptions is an interesting open problem.

To the best of our knowledge, this is the first paper that assesses whether an easily computed sequence performs provably well, rather than trying to find the optimal sequence for a special (typically low-dimensional) instance, or comparing heuristics through simulation. Finding the optimal sequence is an important (but  inherently difficult) problem, and we hope our approach triggers more research in this direction.


\appendix

\section{Proofs corresponding to Section~\ref{sec:bounds}}\label{App3}

Here we prove Lemma~\ref{UB} and Lemma~\ref{LB}, that are used in Section~\ref{sec:bounds}.

\begin{customlemma}{\ref{UB}}
Under Assumption \ref{ASymm}, 
\begin{equation*}
\E W_{k+1} \leq  \E\left(\stp{\tau(1)} + \cdots + \stp{\tau(k)} \right)^+
                  \;+\; \E\left(\stp{\tau(1)} + \cdots + \stp{\tau(k-1)}\right)^+.
\end{equation*}
\end{customlemma}

\begin{proof}
By Lemma~\ref{reflectedUB}, $W_k$ is stochastically dominated by $|X_{\tau(1)}+\dots+X_{\tau(k-1)}|$. As a consequence,
\begin{equation}\label{improveUB}
    \E W_{k+1} = \E \left(W_k+\stp{\tau(k)}\right)^+
               \leq  \E\left(|\stp{\tau(1)} + \cdots + \stp{\tau(k-1)}|+\stp{\tau(k)}\right)^+.
\end{equation}

If $Y$ and $Z$ are independent and both have symmetric distributions, then
\begin{equation*}
\E(|Y|+Z)^+ = \E (Y+Z)^++\E Y^+.
\end{equation*}
This is easily checked by conditioning on $|Y|=a$ and $|Z|=b$, as then $Y$ is either $a$ or $-a$ with probability $\frac12$, and similarly for $Z$. Applying this result to the upper bound in (\ref{improveUB}), we find the upper bound in the lemma.
\end{proof}

\begin{customlemma}{\ref{LB}}
Under Assumption \ref{ASymm}, for any $\ell$,
\begin{align*}
\E W_{k+1} \geq  \tfrac12\bigg(\E\left(\stp{\tau(1)} + \cdots + \stp{\tau(k)}\right)^+ 
+ \E\left(\stp{\tau(1)} + \cdots + \stp{\tau(\ell)}\right)^+ \\
+\E\left(\stp{\tau(\ell+1)} + \cdots + \stp{\tau(k)}\right)^+\bigg).
\end{align*}
\end{customlemma}

\begin{proof}
Let $S_{\ell}'= \stp{\tau(1)}+\stp{\tau(2)} + \cdots + \stp{\tau(\ell)}$, so that $S_{\ell}'= S_k-S_{k-\ell}$. 
As $W_{k+1}=\max\{0,S_1,\dots,S_k\}$, we then have
\begin{equation}\label{improveLB}
W_{k+1} \geq  \max\{0,S_{k-\ell},S_k\}=\max\{0,S_{k-\ell},S_{k-\ell}+S_{\ell}'\}=(S_{k-\ell}+(S_{\ell}')^+)^+.
\end{equation}

If $Y$ and $Z$ are independent and both have symmetric distributions, then
\begin{equation*}
\E(Y+Z^+)^+ = \tfrac12\left(\E (Y+Z)^++\E Y^++\E Z^+\right).
\end{equation*}
Again, this is easily checked by conditioning on $|Y|=a$ and $|Z|=b$, as then $Y$ is either $a$ or $-a$ with probability $\frac12$, and similarly for $Z$. Applying this result to the lower bound in (\ref{improveLB}), we find the lower bound in the lemma.
\end{proof}


\section{Proofs corresponding to Section~\ref{sec:asymptotic}}\label{App5}

In this appendix we prove Proposition~\ref{UB2}, Proposition~\ref{LB2}, and hence deduce Theorem~\ref{!}, building on the ideas presented in Section~\ref{sec:asymptotic}. 

\begin{proofof}{Proposition \ref{UB2}.}
For ease of notation, we assume throughout this proof that patients are renumbered so that $\tau(i) = i$.
Note that
$
\Sigma_{k}^2 := \sum_{i=1}^k \sigma_{n,i}^2
$
is the variance of both $S_{k}$ and $\tilde{S}_{k}(a)$.
In order to apply Theorem~\ref{CLT} we scale all steps, and hence $S_{k}$ and $\tilde{S}_{k}(a)$, by a factor $1/\Sigma_{k}$. 
For both martingales the squared increments $(\stpn{k-i+1})^2$ are independent of the previous increments, so that after rescaling the conditional variance after $k$ steps equals one for both martingales. Note that we can recognize the $q_k$ from Assumption~\ref{assumption} in the upper bound, so we find for any $x$ that
\begin{align}\label{CLT1}
&\Bigg|\,\p\left(\frac{S_{k}}{\Sigma_{k}}>x\right)-\p(Z>x)\,\Bigg|\,\leq  C_\delta q_k^{1/(3+\delta)},\\
\label{CLT2}
&\Bigg|\,\p\left(\frac{\tilde{S}_{k}(a)}{\Sigma_{k}}>x\right)-\p(Z>x)\,\Bigg|\,\leq  C_\delta q_k^{1/(3+\delta)}.
\end{align}

Using inequality \eqref{CLT1} and Chebyshev's inequality, we find
\begin{align*}
\frac{1}{\Sigma_{k}}\int_0^\infty \p(S_{k}>a)\mathrm{d}a &= \int_0^\infty \p\left(\frac{S_{k}}{\Sigma_{k}}>x\right)\mathrm{d}x \\
&\leq   \int_0^{1/\sqrt{q_k}^{1/(3+\delta)}} \left(\p(Z>x)+ C_\delta q_k^{1/(3+\delta)}\right)\mathrm{d}x+\int_{1/\sqrt{q_k}^{1/(3+\delta)}}^\infty  \frac{1}{x^2}\mathrm{d}x \\
&\leq  \E Z^+ +(C_\delta+1)\sqrt{q_k}^{1/(3+\delta)}.
\end{align*}
A similar reasoning using \eqref{CLT2} finds the same upper bound for
\begin{align*}
\frac{1}{\Sigma_{k}}\int_0^\infty \p(\tilde{S}_{k}(a)>a)\,\mathrm{d}a.
\end{align*}
Finally,
\begin{align*}
    \E W_{n,k+1} &\leq  \int_0^\infty \left(\p(S_{k}\geq  a)+\p(\tilde{S}_{k}(a)>a)\right)\mathrm{d}a 
		= \int_0^\infty \left(\p(S_{k}> a)+\p(\tilde{S}_{k}(a)>a)\right)\mathrm{d}a\\
                 &\leq 2\E Z^+ + 2(C_{\delta}+1)q_k^{(3+\delta)/2}.
\end{align*}
\end{proofof}

\begin{proofof}{Proposition \ref{LB2}.}
    We will first lay out the technical groundwork that we need for this proof.

    Fix any $n$ and $k \leq n$ (later we will take limits).
    As in the proof above, assume that $\tau(i) = i$.
Recall from Section~\ref{sec:asymptotic} that $(S'_j)_{j \leq k}$ is the random walk with steps $\stpn{i}\mathbbm{1}_{\stpn{i}\leq  c_{n,k}}$ instead of $\stpn{i}$;
and $\hat{S}_{k}'(a)$, $\tilde{S}_{k}'(a)$ are defined based on this random walk in the same way that
$\hat{S}_{k}(a)$ and $\tilde{S}_{k}(a)$ were defined based on $(S_j)$.
For now, $c_{n,k}$ can take any positive value.
We thus have
\begin{align*}
\p(W_{n,k+1}'\geq  a)\geq  \p(S_{k}'>a)+\p(\hat{S}_{k}'(a)>a)
\geq  \p(S_{k}'>a)+\p(\tilde{S}_{k}'(a)>a+2c_{n,k}).
\end{align*}

Note that the steps no longer have mean zero, and so $\tilde{S}_{j}'(a)$ is no longer a martingale. 
To repair this issue, we must know how much the change in steps due to the indicator affects the mean and variance of all the steps. 
For this we have the following lemma. Define $\Sigma_{n,k}^2 := \Var \, S_k$, $(\Sigma_{n,k}')^2:=\Var\, S_{k}'$ and
$\gamma_{n,k}:=\left(\Sigma_{n,k} / c_{n,k}\right)^\delta$.

\begin{lemma}\label{enumLemma}
    ~\smallskip
\begin{enumerate}[{\normalfont (1)}]
\setlength\itemsep{1em}
\item $\displaystyle\begin{aligned}
        \frac{\sum_{i=1}^k \E\left[(\stpn{i})\mathbbm{1}_{\stpn{i}>c_{n,k}}\right]}{\Sigma_{n,k}}\leq  q_k \gamma_{n,k}^{(1+\delta)/\delta} \, ; \; \text{and} 
\end{aligned}$ \label{meanApp}
\item $\displaystyle\begin{aligned}
	\frac{(\Sigma_{n,k}')^2}{\Sigma_{n,k}^2}\geq  1-2q_k\gamma_{n,k}\,.
\end{aligned}$ \label{VarSApp}
\end{enumerate}
\end{lemma}

\begin{proof}
	\ref{meanApp} Using that $\stpn{i} > c_{n,k}$ implies $|\stpn{i}|^{1+\delta}/c_{n,k}^{1+\delta}>1$, we find
\begin{align*}
    \E\left[(\stpn{i})\mathbbm{1}_{\stpn{i}>c_{n,k}}\right] \leq  \E\left[\left(|\stpn{i}|^{2+\delta}/c_{n,k}^{1+\delta}\right)\mathbbm{1}_{\stpn{i}>c_{n,k}}\right] 
\leq  \E|\stpn{i}|^{2+\delta}/c_{n,k}^{1+\delta}.
\end{align*}
Summing over $i$ and dividing both sides by $\Sigma_{n,k}$ we find
\begin{equation*}
    \frac{\sum_{i=1}^k \E\left[(\stpn{i})\mathbbm{1}_{\stpn{i}>c_{n,k}}\right]}{\Sigma_{n,k}}\leq  \frac{\sum_{i=1}^n \E|\stpn{i}|^{2+\delta}}{\Sigma_{n,k} c_{n,k}^{1+\delta}}
= \frac{\sum_{i=1}^n \E|\stpn{i}|^{2+\delta}}{\Sigma_{n,k}^{2+\delta}}\left(\frac{\Sigma_{n,k}}{c_{n,k}}\right)^{1+\delta} \leq  q_k\gamma_{n,k}^{(1+\delta)/\delta},
\end{equation*}
as was claimed.

   \ref{VarSApp} Analogous to the proof of part~(\ref{meanApp}), we can also deduce that
\begin{align}\label{Var}
\frac{\sum_{i=1}^k \E\left[\stpn{i}^2\mathbbm{1}_{\stpn{i}>c_{n,k}}\right]}{\Sigma_{n,k}^2}\leq  q_k \gamma_{n,k}.
\end{align}
Using that $\E\stpn{i}=0$, we find
\begin{align*}
\frac{(\Sigma_{n,k}')^2}{\Sigma_{n,k}^2}&=\frac{\sum_{i=1}^n \E\left[\stpn{i}^2\mathbbm{1}_{\stpn{i}\leq  c_{n,k}}\right]}{\Sigma_{n,k}^2}
-\frac{\sum_{i=1}^n \left(\E\left[\stpn{i}\mathbbm{1}_{\stpn{i}\leq  c_{n,k}}\right]\right)^2}{\Sigma_{n,k}^2} \\
&= \frac{\Sigma_{n,k}^2-\sum_{i=1}^n \E\left[\stpn{i}^2\mathbbm{1}_{\stpn{i}> c_{n,k}}\right]}{\Sigma_{n,k}^2} 
- \frac{\sum_{i=1}^n \left(\E\left[\stpn{i}\mathbbm{1}_{\stpn{i}> c_{n,k}}\right]\right)^2}{\Sigma_{n,k}^2}.
\end{align*}
Now applying Jensen's inequality to the last term and (\ref{Var}),
\begin{align*}
    \frac{(\Sigma_{n,k}')^2}{\Sigma_{n,k}^2}\geq  \frac{\Sigma_{n,k}^2-2\sum_{i=1}^n \E\left[\stpn{i}^2\mathbbm{1}_{\stpn{i}> c_{n,k}}\right]}{\Sigma_{n,k}^2} \geq  1-2q_k\gamma_{n,k},
\end{align*}
which is what we wanted to prove.
\end{proof}

Now we have the following result. Define, 
with $Z$ a standard normal random variable and 
$C_\delta$ the constant featuring in Theorem~\ref{CLT},
\begin{align*}
D_k:=\int_0^{1/\sqrt{q_k}^{1/(3+\delta)}} \p(Z>x)\mathrm{d}x-C_\delta\sqrt{q_k}^{1/(3+\delta)}.
\end{align*}

\begin{lemma}\label{LB1}
    We have
\begin{align*}
\frac{\E W_{n,k+1}}{\Sigma_{n,k}} \geq 2D_k\sqrt{1- 2q_k \gamma_{n,k}} -\frac{2c_{n,k}}{\Sigma_{n,k}}-2 q_k \gamma_{n,k}^{(1+\delta)/\delta}.
\end{align*}
\end{lemma}

\begin{proof}
Again we want to apply Theorem~\ref{CLT}. This time we not only need to divide by the standard deviation $\Sigma_{n,k}'$ of $S_{n,k}'$, but also subtract the (negative) mean. We then find
\begin{align*}
&\frac{1}{\Sigma_{n,k}'}\int_0^\infty \p(S_{k}'>a)\,\mathrm{d}a =\int_{-\E S_{k}'/\Sigma_{n,k}'}^\infty \p\left(\frac{S_{k}'-\E S_{k}'}{\Sigma_{n,k}'}>x\right)\mathrm{d}x \\
&\geq  \int_0^{1/\sqrt{q_k}^{1/(3+\delta)}} \left(\p(Z>x)-C_\delta q_k^{1/(3+\delta)}\right)\mathrm{d}x 
-\int_0^{-\E S_{k}'/\Sigma_{n,k}'} \p\left(\frac{S_{k}'-\E S_{k}'}{\Sigma_{n,k}'}>x\right)\mathrm{d}x \\
&\geq  D_k+\frac{\E S_{k}'}{\Sigma_{n,k}'},
\end{align*}
where we used in the last step that probabilities are bounded by one.

Next, we need a bound for
\begin{align}\label{BBB}
\frac{1}{\Sigma_{n,k}'}\int_0^\infty \p(\tilde{S}_{k}'(a)>a+2c_{n,k})\mathrm{d}a.
\end{align}
Note that $\tilde{S}_{j}'(a)$ is no longer a martingale, as the mean step size is no longer zero. However, we can remedy this by noting that $\tilde{S}_{k}'(a)$ is bounded below by
\begin{align*}
    \tilde{S}_{k}(a)-\sum_{i=1}^k (\stpn{i})\mathbbm{1}_{\stpn{i}>c_{n,k}}=\tilde{S}_{k}(a)-S_{k}+S_{k}',
\end{align*}
and $\tilde{S}_{j}(a)-S_{j}+S_{j}'-\E S_{j}'$ is again a martingale with mean zero. Applying Theorem~\ref{CLT} to this martingale and taking into account the shift $\E S_{k}'$ and overshoot $2c_{n,k}$, we find   that \eqref{BBB} is bounded below by
\begin{align*}
D_k +\frac{\E S_{k}'}{\Sigma_{n,k}'}-\frac{2c_{n,k}}{\Sigma_{n,k}'}.
\end{align*}
Now adding up these two lower bounds, we find
\begin{align*}
\frac{\E W_{n,k+1}'}{\Sigma_{n,k}'} \geq 2D_k+\frac{2\,\E S_{k}'}{\Sigma_{n,k}'}  -\frac{2c_{n,k}}{\Sigma_{n,k}'}.
\end{align*}
Now note that, due to  Lemma~\ref{enumLemma} \ref{VarSApp},
${\Sigma_{n,k}'}/{\Sigma_{n,k}}\geq  \sqrt{1- 2q_k \gamma_{n,k}}.$
In addition, $\E W_{n,k+1}\geq  \E W_{n,k+1}'$. Consequently,
\begin{align*}
\frac{\E W_{n,k+1}}{\Sigma_{n,k}} \geq  2D_k\sqrt{1- 2q_k \gamma_{n,k}}
+\frac{2\,\E S_{k}'}{\Sigma_{n,k}} -\frac{2c_{n,k}}{\Sigma_{n,k}} .
\end{align*}
Applying Lemma~\ref{enumLemma} \ref{meanApp},
\begin{align*}
\frac{\E W_{n,k+1}}{\Sigma_{n,k}} \geq  2D_k\sqrt{1- 2q_k \gamma_{n,k}}-2 q_k \gamma_{n,k}^{(1+\delta)/\delta}-\frac{2c_{n,k}}{\Sigma_{n,k}},
\end{align*}
which is the bound we wanted to prove.
\end{proof}

    We are now finally ready to proceed with the proof of Proposition~\ref{LB2}.
    The freedom remains to choose $c_{n,k}$, which we have only assumed to be positive.
We would like to have $c_{n,k}/\Sigma_{n,k}\to 0$,  $q_k\gamma_{n,k}^{(1+\delta)/\delta}\to 0$ and $q_k\gamma_{n,k}\to 0$ as $k\to\infty$. A choice that achieves this goal is
\begin{align*}
c_{n,k}:= \sqrt{q_k}^{1/(\delta+1)}\Sigma_{n,k},
\end{align*}
so that
$
\gamma_{n,k}=q_k^{-\delta/(2\delta+2)}. $
We thus obtain, with $\bar q_k:=\sqrt{q_k}^{1/(\delta+1)}+\sqrt{q_k}$,
\begin{align*}
\frac{\E W_{n,k+1}}{\Sigma_{n,k}} \geq  
2\left(\int_0^{1/\sqrt{q_k}^{1/(3+\delta)}} \p(Z>x)\mathrm{d}x-C_\delta\sqrt{q_k}^{1/(3+\delta)}\right)\sqrt{1-2\sqrt{q_k}^{2-\delta/(\delta+1)}}-2\bar q_k.
\end{align*}
Note that this converges to $\E|Z|$ as $q_k\to0$, so this completes the proof of Proposition~\ref{LB2}.
\end{proofof}

\begin{proofof}{Theorem~\ref{!}.}
    We first state an immediate consequence of combining Proposition~\ref{UB2} and Proposition~\ref{LB2}.
We use $\Wsvf_{n,i}$ and $\Isvf_{n,i}$ to refer specifically to the waiting and idle times for the \svf{} sequence, and $\Wopt_{n,i}$ and $\Iopt_{n,i}$ for the optimal sequence.

\begin{lemma}\label{enkel}
Under Assumption $\ref{assumption}$,   
for any $\varepsilon>0$ there exists a $K$ depending on $\varepsilon$ only, such that for all $k\geq  K$ and for all $n\geq  k-1$,
\begin{align*}
\E \Wsvf_{n,k} \leq  (1+\varepsilon)\E \Wopt_{n,k}.
\end{align*}
\end{lemma}

\begin{proof}
By Proposition~\ref{LB2}, for any $\varepsilon>0$, we can choose $k$ sufficiently large such that
\begin{align*}
\E W_{n,k+1} \geq  (1-\varepsilon)\sqrt{\sum_{i=1}^k \sigma_{n,\tau(i)}^2} \cdot \E|Z|\geq(1-\varepsilon)\sqrt{\sum_{i=1}^k \sigma_{n,i}^2} \cdot \E|Z|
\end{align*}
for any sequence $\tau$, in particular for the optimal sequence. Here we also used that $\sigma_{n,1}^2,\dots, \sigma_{n,k}^2$ are the $k$ smallest variances.
For the \svf{} sequence, Proposition~\ref{UB2} gives us that for sufficiently large $k$ we have
\begin{align*}
\E \Wsvf_{n,k+1} \leq  (1+\varepsilon)\sqrt{\sum_{i=1}^k \sigma_{n,i}^2}\cdot \E|Z|.
\end{align*}
Combining these two bounds completes the proof.
\end{proof}

By the Lindley recursion we have
\begin{align*}
    \E W_{n,k+1}=\E(W_{n,k}+\stpn{\tau(k)})^+
\geq  \E(W_{n,k}+\stpn{\tau(k)})=\E W_{n,k},
\end{align*}
so, for any $n$ and $\tau$, $\E W_{n,k}$ is increasing in $k$.  With $K$ as in Lemma~\ref{enkel}, we consequently have
\begin{align*}
\sum_{k=1}^n \E \Wsvf_{n,k} \leq  K\E \Wsvf_{n,K}+\sum_{k=K+1}^n \E \Wsvf_{n,k} 
\leq  (1+\varepsilon)\left(K\E \Wopt_{n,K}+\sum_{k=K+1}^n \E \Wopt_{n,k}\right).
\end{align*}
Again using that $\E W_{n,k}$ is increasing in $k$, we also have
\begin{align*}
\left.K\,\E \Wopt_{n,K}\,\right/\sum_{k=K+1}^n \E \Wopt_{n,k}\leq  \frac{K}{n-K},
\end{align*}
which is smaller than $\varepsilon$ for sufficiently large $n$. Then we have
\begin{align*}
\sum_{k=1}^n \E \Wsvf_{n,k}&\leq  (1+\varepsilon)\left(K\E \Wopt_{n,K} + \sum_{k=K+1}^n \E \Wopt_{n,k}\right) \\
&\leq  (1+\varepsilon)^2 \sum_{k=K+1}^n \E \Wopt_{n,k} 
\leq  (1+\varepsilon)^2 \sum_{k=1}^n \E \Wopt_{n,k}.
\end{align*}
Recall that the total expected idle time is equal to $\E W_{n,n}$.
Since also $\E \Wsvf_{n,n} \leq (1+\varepsilon)\Wopt_{n,n}$ for $n \geq K$, 
\begin{align*}
    \omega \E \Wsvf_{n,n} +(1-\omega)\sum_{k=1}^n \E \Wsvf_{n,k} \leq  (1+\varepsilon)^2 \left(\omega \E \Wopt_{n,n} +(1-\omega)\sum_{k=1}^n \E \Wopt_{n,k}\right)
\end{align*}
for $n$ sufficiently large (independent of $\omega$).
As this holds for any $\varepsilon>0$, we find $\varrho_{\omega}(\stv{n}) \to 1$ as $n\to\infty$, and Theorem~\ref{!} is proved.
\end{proofof}


\section{Proofs corresponding to Section~\ref{sec:combined}}\label{App4} 

\subsection{Proofs corresponding to Section 5.1 and 5.3}

Here we prove Proposition~\ref{SVF}, Proposition~\ref{PropFreeLB}, and Theorem~\ref{rLognormal} from Section~\ref{sec:combined}.
In order to prove Proposition~\ref{SVF}, we need a number of lemmas. We will denote the waiting times and idle times under the \svf{} sequence with schedule $\bm{x}=\bm{\mu}+\alpha\bm{\sigma}$ by $\Wsvf_i$ and $\Isvf_i$ respectively, and we will denote the waiting times and idle times under the optimal combination of sequence and schedule by $\Wopt_i$ and $\Iopt_i$.

\begin{lemma}\label{expand}
Let $M$ be the maximum of a random walk with steps $Y_1,\dots,Y_k$. Now let $i\in\{1,\ldots,k\}$, and let $c\geq 1$. Let $M'$ be defined as the maximum of a random walk with steps $Y_1,\ldots, Y_i, c\,Y_{i+1},\ldots,c\,Y_k$. Then $M\leq  M'$.
\end{lemma}

\begin{proof}
Suppose that the maximum of the first random walk is attained at time $j$, that is, $M=Y_1+\dots+Y_j$. If $j\leq  i$, then the second random walk also attains the value $Y_1+\dots+Y_j$, so $M\leq  M'$. If $j>i$, then the second random walk attains the value $Y_1+\dots+Y_i+cY_{i+1}+\dots+cY_j$. Now $Y_{i+1}+\dots+Y_j\geq  0$, as otherwise $Y_1+\dots+Y_i>Y_1+\dots+Y_j$, in contradiction with $M$ being the maximum. From this and $c\geq 1$ it follows that
\begin{align*}
Y_1+\dots+Y_i+c\,Y_{i+1}+\dots+c\,Y_j\geq  Y_1+\dots+Y_j,
\end{align*} 
so $M\leq  M'$.
\end{proof}

\begin{lemma}\label{max}
Suppose Assumption $\ref{family}$ holds, and we use the schedule $\bm{x}=\bm{\mu}+\alpha\bm{\sigma}$.
Let $M_k$ be the all-time maximum of a random walk with i.i.d.\ steps distributed as $\sigma_k(B-\alpha)$. Then $\Wsvf_{k+1}$ is stochastically dominated by $M_k$, for all $k$.
\end{lemma}

\begin{proof}
By (\ref{W}), we know that $\Wsvf_{k+1}$ is the maximum of a random walk with steps $B_k-x_k,B_{k-1}-x_{k-1},\dots,B_1-x_1$.
Now note that $B_i\dis\mu_i+\sigma_i B$ and $x_i=\mu_i+\alpha\sigma_i$, hence $B_i-x_i\dis \sigma_i(B-\alpha)$. 
So $\Wsvf_{k+1}$ can be represented as the maximum of a random walk with steps distributed as $\sigma_k(B-\alpha),\dots,\sigma_1(B-\alpha)$.

We first multiply the last step of this random walk $\sigma_2/\sigma_1$. By Lemma~\ref{expand}, we then see that $W_{k+1}$ is stochastically dominated by the maximum of a random walk with steps distributed as $\sigma_k(B-\alpha),\dots,\sigma_2(B-\alpha),\sigma_2(B-\alpha)$. 
The next step is to multiply the last two steps with $\sigma_3/\sigma_2$. Again by Lemma~\ref{expand},  $\Wsvf_{k+1}$ is stochastically dominated by the maximum of a random walk with steps $\sigma_k(B-\alpha),\dots,\sigma_3(B-\alpha)$, $\sigma_3(B-\alpha)$, $\sigma_3(B-\alpha)$. Continuing in this way, we find that $\Wsvf_{k+1}$ is stochastically dominated by the maximum of a random walk with $k$ steps distributed as $\sigma_k(B-\alpha)$.

Now note that adding extra steps to a random walk can only increase the maximum. Therefore, we now find that $\Wsvf_{k+1}$ is stochastically dominated by $M_k$.
\end{proof}

\begin{lemma}\label{UBW}
Under Assumption $\ref{family}$, when using the schedule $\bm{x}=\bm{\mu}+\alpha\bm{\sigma}$, we have for all $k$ that
\begin{align*}
\E \Wsvf_{k+1}\leq  \frac{\sigma_k}{2\alpha}.
\end{align*}
\end{lemma}

\begin{proof}
We  need the following bound by \cite{Kingman}.
Let $M$ be the all-time maximum of a random walk with i.i.d.\ steps distributed as $Y$, with $\E Y<0$. Then
\begin{align*}
\E M\leq  \frac{\Var(Y)}{2|\E Y|}.
\end{align*}
Let $M_k$ be as in Lemma \ref{max}. Then, by Lemma~\ref{UBW} and Kingman's bound, we have
\begin{align*}
\E \Wsvf_{k+1}\leq  \E M_k \leq  \frac{\Var(\sigma_k(B-\alpha))}{2|\E(\sigma_k(B-\alpha))|}=\frac{\sigma_k^2}{2\sigma_k\alpha}=\frac{\sigma_k}{2\alpha},
\end{align*}
as claimed.
\end{proof}

Now we are ready to prove Proposition~\ref{SVF}.

\begin{customprop}{\ref{SVF}}
Suppose $\alpha$ is given by~$(\ref{alphaatje})$.
Under Assumption $\ref{family}$,
\begin{align*}
\Obj{\Id}{\bm{\mu}+\alpha\bm{\sigma}}{\omega}\leq  \sqrt{2\omega}\sum_{i=1}^{n-1} \sigma_i.
\end{align*}
\end{customprop}

\begin{proof}
We already have a bound on the mean waiting time in Lemma~\ref{UBW}, so we proceed by considering the mean idle time. 
Taking expectations in (\ref{I}) and using the fact that $x_i=\mu_i+\alpha\sigma_i$,
\begin{align*}
\sum_{i=1}^n \E \Isvf_i+\sum_{i=1}^n \mu_i =\sum_{i=1}^n \mu_i+\alpha\sum_{i=1}^{n-1} \sigma_i+\E \Wsvf_n.
\end{align*}
Hence, by virtue of Lemma~\ref{UBW},
\begin{align*}
\sum_{i=1}^n \E \Isvf_i =\alpha \sum_{i=1}^{n-1}\sigma_i+\E \Wsvf_n\leq  \alpha \sum_{i=1}^{n-1} \sigma_i +\frac{\sigma_{n-1}}{2\alpha}.
\end{align*}
Upon combining the bounds for the mean waiting times and the total mean idle time, we find that
\begin{align*}
\Obj{\Id}{\bm{\mu}+\alpha\bm{\sigma}}{\omega} \leq  \alpha \omega\sum_{i=1}^{n-1} \sigma_i +\frac{\omega}{2\alpha} \sigma_{n-1} +\frac{1-\omega}{2\alpha}\sum_{i=1}^{n-1} \sigma_i.
\end{align*}
Through standard calculus we find that this upper bound is minimized for the $\alpha$ given in (\ref{alphaatje}).
The corresponding upper bound for this $\alpha$ is
\begin{align*}
\sqrt{2\omega(1-\omega)}\sqrt{\sum_{i=1}^{n-1}\sigma_i \left(\sum_{i=1}^{n-1}\sigma_i+\frac{\omega}{1-\omega}\sigma_{n-1}\right)}.
\end{align*}
Using the fact that $\sigma_{n-1}\leq  \sum_{i=1}^{n-1}\sigma_i$, we find the upper bound in Proposition~\ref{SVF}.
\end{proof}

In order to prove Proposition~\ref{PropFreeLB}, we need the following lemma, which is known from the classical newsvendor problem (see e.g.\ \cite{Khouja}). 

\begin{lemma}\label{quantile}
Let $X$ be a random variable, with $Q_X$ its quantile function. Then for $\omega\in(0,1)$,
\begin{align*}
\omega\E(X-c)^- +(1-\omega) \E(X-c)^+
\end{align*}
is minimal for $c=Q_X(1-\omega)=\inf\{x:1-\omega\leq \p(X\leq  x)\}$.
\end{lemma}

\begin{customprop}{\ref{PropFreeLB}}
Under Assumption $\ref{family}$, for any sequence and schedule,  
\begin{equation*}
\Obj{\tau}{\bm{x}}{\omega}
 \geq  \left[\omega\E B(\omega)^- +(1-\omega)\E B(\omega)^+ \right]\sum_{i=1}^{n-1} \sigma_i.
\end{equation*}
\end{customprop}

\begin{proof}
Consider some arbitrary sequence $\tau$ and schedule $\bm{x}$.
Recall that the idle and waiting times satisfy the recursions in (\ref{LindleyModel}).
We consider 
\begin{align}\label{Lindley}
\omega \E I_{k+1}+(1-\omega)\E W_{k+1} = \text{ }\omega\E\left(W_k+B_{\tau(k)}-x_{\tau(k)}\right)^-  +(1-\omega) \E \left(W_k+B_{\tau(k)}-x_{\tau(k)}\right)^+. 
\end{align}
Now note that $W_k+B_{\tau(k)}-x_{\tau(k)}=B_{\tau(k)}-(x_{\tau(k)}-W_k)$, so by minimizing (\ref{Lindley}) over all possible values of $x_{\tau(k)}-W_k$, we find by Lemma~\ref{quantile} that
\begin{align*}
\omega\E I_{k+1} +(1-\omega)\E W_{k+1}\geq  \text{ }\omega\E(B_{\tau(k)}-Q_{B_{\tau(k)}}(1-\omega))^-  +(1-\omega)\E(B_{\tau(k)}-Q_{B_{\tau(k)}}(1-\omega))^+.
\end{align*}

Because $B_{\tau(k)}\dis \mu_{\tau(k)}+\sigma_{\tau(k)}B$, it follows that
$
Q_{B_{\tau(k)}}(1-\omega)=\mu_{\tau(k)}+\sigma_{\tau(k)}Q_B(1-\omega).
$
Then 
\begin{align*}
B_{\tau(k)}-Q_{B_{\tau(k)}}(1-\omega)\dis \mu_{\tau(k)}+\sigma_{\tau(k)}B-(\mu_{\tau(k)}+\sigma_{\tau(k)}Q_B(1-\omega))  = \sigma_{\tau(k)}(B-Q_B(1-\omega)),
\end{align*}
which, recalling that $B(\omega) = B-Q_B(1-\omega)$, leads to
\begin{align*}
\omega\E I_{k+1} +(1-\omega)\E W_{k+1} \geq  \sigma_{\tau(k)}\left[\omega\E B(\omega)^- +(1-\omega)\E B(\omega)^+\right].
\end{align*}

Note that $\sum_{i=1}^{n-1} \sigma_{\tau(i)}\geq  \sum_{i=1}^{n-1}\sigma_i$, as the $\sigma_i$ were put  in increasing order. Now summing over $k$ we find the lower bound of Proposition~\ref{PropFreeLB}.
\end{proof}

To prove Theorem~\ref{rLognormal} we need the following lemma, which fulfills the same role for lognormally distributed service times as Lemma~\ref{max} does for location-scale families. We then prove the upper and lower bounds needed for Theorem~\ref{rLognormal} in two propositions.

\begin{lemma}\label{maxLog}
Suppose the $B_i$ are lognormally distributed with $m_1\leq \dots \leq m_n$ and $s_1^2\leq \dots\leq s_n^2$, and that we use the schedule $\bm{x}=(1+\alpha)\bm{\mu}$. Let $M_k$ be the all-time maximum of  a random walk with i.i.d.\ steps distributed as $B_k-x_k$. Then $\E \Wsvf_{k+1} \leq \E M_k$.
\end{lemma}

\begin{proof}
Let $M_k^{(i)}$ be the maximum of a random walk with steps distributed as $B_k-x_k,B_{k-1}-x_{k-1},\dots,B_{i+1}-x_{i+1}$ followed by $i$ steps distributed as $B_i-x_i$. We will prove that $\E M_k^{(i)}\leq \E M_k^{(i+1)}$ for $i=1,\dots,k-1$. As $\E \Wsvf_{k+1}=\E M_k^{(1)}$ and $\E M_k^{(k)}\leq \E M_k$ (adding extra steps increases the maximum), it then follows that $\E \Wsvf_{k+1} \leq \E M_k$.

Consider the random walk with steps distributed as $B_k-x_k,B_{k-1}-x_{k-1},\dots,B_{i+1}-x_{i+1}$, followed by $i$ steps distributed as $B_i-x_i$. Let $Z$ be a standard normal random variable. Note that
\begin{align*}
B_i-x_i=B_i-(1+\alpha)\mu_i &\dis \exp(m_i+s_iZ)-(1+\alpha)\exp(m_i+s_i^2/2)\\
&=\exp(m_i)\left[\exp(s_iZ)-(1+\alpha)\exp(s_i^2/2)\right].
\end{align*}
Use Lemma~\ref{expand} to show we get an upper bound on $M_k^{(i)}$ by replacing all steps distributed as $B_i-x_i$ by steps distributed as
\begin{align*}
X_i':=\exp(m_{i+1}+(s_{i+1}^2-s_i^2)/2)\left[\exp(s_iZ)-(1+\alpha)\exp(s_i^2/2)\right].
\end{align*}

Let $Z'$ be another standard normal random variable independent of $Z$. Then
\begin{align*}
B_{i+1}-x_{i+1}\dis \exp(m_{i+1}+s_iZ+\epsilon Z')-(1+\alpha)\exp(m_{i+1}+(s_i^2+\epsilon^2)/2)
\end{align*}
with $\epsilon:=\sqrt{s_{i+1}^2-s_i^2}$, 
and
\begin{align*}
&\E[ \exp(m_{i+1}+s_iZ+\epsilon Z')-(1+\alpha)\exp(m_{i+1}+(s_i^2+\epsilon^2)/2)|X_i']\\
&=\E[\exp(m_{i+1}+s_iZ+\epsilon Z')-(1+\alpha)\exp(m_{i+1}+(s_i^2+\epsilon^2)/2)|Z]\\
&=\exp(m_i+\epsilon^2/2)\left[\exp(s_iZ)-(1+\alpha)\exp(s_i^2/2)\right]= X_i'.
\end{align*}
It follows by Lemma~\ref{condExp} that $X_i'\cleq B_{i+1}-x_{i+1}$. As the maximum of a random walk is a convex function in each of the individual stepsizes, we can replace each step distributed as $X_i'$ by one distributed as $B_{i+1}-x_{i+1}$. Therefore, $\E M_k^{(i)}\leq \E M_k^{(i+1)}$, which completes the proof.
\end{proof}

\begin{proposition}\label{LognormalUB}
Suppose the $B_i$ are lognormally distributed with $m_1\leq \dots\leq m_n$ and $s_1^2\leq \dots\leq s_n^2$. Suppose $\alpha$ is given by
\begin{align*}
\alpha=\frac{1}{\sqrt{2\omega}} \sqrt{(\exp(s_{n-1}^2)-1)}.
\end{align*}
Then
\[
\Obj{\Id}{(1+\alpha)\bm{\mu}}{\omega}\leq 2\omega \alpha\sum_{i=1}^{n-1}\exp(m_i+s_i^2/2).
\]
\end{proposition}

\begin{proof}
With $M_k$ as in Lemma~\ref{maxLog}, we can now apply Kingman's bound to find
\begin{align*}
\E\Wsvf_{k+1}&\leq \E M_k\leq \frac{\Var(B_k)}{2|\E B_k-(1+\alpha)\E B_k|}=\frac{1}{2\alpha}(\exp(s_k^2)-1)\exp(m_k+s_k^2/2)\\
&\leq \frac{1}{2\alpha}(\exp(s_{n-1}^2)-1)\exp(m_k+s_k^2/2).
\end{align*}
By equation (\ref{I}) we also have
\begin{align*}
\sum_{i=1}^n \E \Isvf_i &=\alpha\sum_{i=1}^{n-1} \E B_i+\E \Wsvf_n \\
&\leq \alpha\sum_{i=1}^{n-1}\exp(m_i+s_i^2/2)+\frac{1}{2\alpha}(\exp(s_{n-1}^2)-1)\exp(m_{n-1}+s_{n-1}^2/2)\\
&\leq \alpha\sum_{i=1}^{n-1}\exp(m_i+s_i^2/2)+\frac{1}{2\alpha}(\exp(s_{n-1}^2)-1)\sum_{i=1}^{n-1}\exp(m_i+s_i^2/2).
\end{align*}
In total, we then find
\begin{align*}
\Obj{\Id}{(1+\alpha)\bm{x}}{\omega}
\leq \omega\alpha\sum_{i=1}^{n-1}\exp(m_i+s_i^2/2)+\frac{1}{2\alpha}(\exp(s_{n-1}^2)-1)\sum_{i=1}^{n-1}\exp(m_i+s_i^2/2).
\end{align*}
Minimizing this upper bound over $\alpha$, we obtain the result.
\end{proof}

\begin{proposition}\label{LognormalLB}
Suppose the $B_i$ are lognormally distributed with $m_1\leq \dots\leq m_n$ and $s_1^2\leq \dots\leq s_n^2$. Then, for any sequence and schedule,
\begin{align*}
\Obj{\tau}{\bm{x}}{\omega} 
\geq \left[(1-\omega)\p(Z\geq Q_Z(1-\omega)-s_1)-\omega\p(Z\leq Q_Z(1-\omega)-s_1)\right]\sum_{i=1}^{n-1} \exp(m_i+s_i^2/2).
\end{align*}
\end{proposition}

\begin{proof}
Similar as in the location-scale family case, we can use Lemma~\ref{quantile} to find
\begin{align*}
\omega \E I_{k+1}+(1-\omega)\E W_{k+1}&=\omega \E(B_{\tau(k)}-(x_{\tau(k)}-W_k))^-+(1-\omega)\E(B_{\tau(k)}-(x_{\tau(k)}-W_k))^+ \\
&\geq \omega\E(B_{\tau(k)}-Q_{B_{\tau(k)}}(1-\omega))^- +(1-\omega)\E(B_{\tau(k)}-Q_{B_{\tau(k)}}(1-\omega))^+.
\end{align*} 
Computing this lower bound, we find with $Z$ being a standard normal random variable that
\begin{align*}
&\omega \E I_{k+1}+(1-\omega)\E W_{k+1}\\
&\geq \exp(m_{\tau(k)}+s_{\tau(k)}^2/2) \left[(1-\omega)\p(Z\geq Q_Z(1-\omega)-s_{\tau(k)})-\omega\p(Z\leq Q_Z(1-\omega)-s_{\tau(k)})\right].
\end{align*}
It can be easily seen that
\begin{align*}
(1-\omega)\p(Z\geq Q_Z(1-\omega)-s_{\tau(k)})-\omega\p(Z\leq Q_Z(1-\omega)-s_{\tau(k)})
\end{align*}
is an increasing function in $s_{\tau(k)}$, and so is minimal for $s_{\tau(k)}=s_1$. Using this and summing over $k$, we find the lower bound of the proposition.
\end{proof}

\begin{proofof}{Theorem~\ref{rLognormal}.}
This follows directly from Proposition~\ref{LognormalUB} and Proposition~\ref{LognormalLB}.
\end{proofof}


\subsection{Proofs corresponding to Section \ref{AsymptoticExample}}\label{AsExampleApp} 

We proceed by providing a detailed argumentation behind the results presented in Section \ref{AsymptoticExample}.

\begin{customlemma}{\ref{lem:svfinterarrival}}
    For the \svf{} sequence in the limit $c\to\infty$, the optimal interarrival times for groups $1$ and $2$ are $x^{(1)}=\frac12\sqrt2$ and $x^{(2)}=a$ respectively.
    Patients in group $1$ incur an average expected waiting time of $\E W^{(1)} = \tfrac12\sqrt{2}$, 
    whereas patients in group $2$ incur no waiting time.
    This yields an average total cost of $\tfrac12(a+\sqrt2)$.
\end{customlemma}

\begin{proof}
First, consider group 1.
Crucially, Kingman's bound is known to be tight for this distribution of $B^{(1)}$ as $c\to\infty$ (\cite{Daley}).
Thus
\begin{align*}
    \lim_{c\to\infty} \E W^{(1)} = \frac{\Var\; B^{(1)}}{2x^{(1)}} = \frac{1}{2x^{(1)}}.
\end{align*}
The average cost for patients in group 1 is thus
\[
    \E I^{(1)} + \E W^{(1)} = x^{(1)} + \frac1{2x^{(1)}},
\]
which is minimized by the choice $x^{(1)}=\frac12\sqrt2$.

Now consider group 2.
Note that $W^{(2)}$ is the maximum of a random walk with steps distributed as $B^{(2)}-x^{(2)}$. 
As a lower bound, we take the value of the random walk right before the step where $B^{(2)}$ takes value $-a$ for the first time. 
The number of steps we take then has a geometric distribution with probability $\frac12$ (i.e.\ the probability of taking $k$ steps is $(1/2)^{k+1}$ for $k=0,1,2,\dots$), which has mean 1, and each of these steps has value $a-x^{(2)}$. Therefore, we find
\begin{align*}
    \E I^{(2)} + \E W^{(2)}\geq x^{(2)}+ (a-x^{(2)}) = a.
\end{align*}
For $x^{(2)}=a$, this lower bound is achieved, as all steps $B^{(2)}-a$ can never be positive and so the waiting time is zero.
In conclusion, the minimal average cost as $c\to\infty$ is $\tfrac12(a+\sqrt2)$, attained with $x^{(1)}=\frac12\sqrt2$ and $x^{(2)}=a$.
\end{proof}

\begin{customlemma}{\ref{lem:bothsame}}
$\lim_{c\to\infty} \E V^{(1)}=\lim_{c\to\infty} \E V^{(2)}$.
\end{customlemma}

\begin{proof}
By the Lindley recursion we have $\E V^{(2)}=\E(V^{(1)}+B^{(1)}-x^{(1)})^+$, 
and conditioning on the two values that $B^{(1)}$ can take we then find
\begin{align*}
\lim_{c\to\infty} \E V^{(2)} &=\lim_{c\to\infty} (1-p)\E \left(V^{(1)}-\frac1c-x^{(1)}\right)^+ +p\E\left(V^{(1)}+c-x^{(2)}\right)^+\\
&\leq \lim_{c\to\infty} \E V^{(1)} +pc =\lim_{c\to\infty} \E V^{(1)}+ \frac{c}{1+c^2}=\lim_{c\to\infty} \E V^{(1)}.
\end{align*}
We also have $\E V^{(2)}\geq \E V^{(1)}$, since $V^{(1)}$ results from the same random walk but with an extra step $B^{(2)}-x^{(2)}\leq 0$ in the beginning.
\end{proof}

\begin{customlemma}{\ref{lem:simplerrw}}
    $\lim_{c \to \infty} \E V^{(1)} \leq \lim_{c \to \infty} \E \tilde{V}$.
\end{customlemma}

\begin{proof}
    Let $M^{(1)}$ be the maximum of $S^{(1)}$ over the first $2c$ steps only.
    We will first show that $\E [V^{(1)} - M^{(1)}] \to 0$ as $c \to \infty$.
    Recall that by Kingman's bound, we have the rough estimate $\E V^{(1)} \leq K$, where $K = ({1+a^2})/({2a+\sqrt{2}})$.
    Let $E$ be the event that $S^{(1)}_{2c} \leq -2c$, {and let $\bar{E}$ be its complement.}
    By Chebyshev's inequality,
    \[ \p(\bar{E}) = \p\left(S^{(1)}_{2c} - \E S^{(1)}_{2c} > c(2a+\sqrt{2} - 2)\right) \leq \frac{2c(1+a^2)}{(c(2a+\sqrt{2}-2))^2} = O(c^{-1}). \]
    Let $Z$ be the maximum of the random walk obtained by removing the first $2c$ steps from $(S^{(1)})_{j \geq 0}$.
    Clearly $Z$ has the same law as $V^{(1)}$, {and} $V^{(1)} - M^{(1)} \leq Z$, so that
    \[ \E\left[V^{(1)} - M^{(1)}\, \big|\, \bar{E}\,\right] \leq \E Z \leq K. \]

    Condition now on $E$ occurring.
    Then $V^{(1)} - M^{(1)} \leq (Z - 2c)^+$.
    Let $R$ be the event that $Z \geq c$. {When the random walk defining $Z$ first crosses level $c$, it will still be below level $2c$. As the random walk after crossing level $c$ has the same law as the random walk defining $Z$, we find that}
    \[
        \E\left[(Z-2c)^+\right] = \E\left[(Z - 2c)^+ | R\right]\p(R)\leq \E[Z]\,\p(R) \leq K\, \p(R).
    \]
    By Markov's inequality, $\p(R) \leq \E Z / c \leq K / c$.
    Combining the above,
    \[ \E\left[V^{(1)} - M^{(1)}\right] \leq K \p(\bar{E}) + K\p(R) \p(E) = O(c^{-1}),
    \]
    which vanishes as $c\to\infty$, as required.

    Let $\tilde{M}$ be the maximum of $\tilde{S}$ over the first $2c$ steps.
It now suffices to show that $\E[M^{(1)} - \tilde{M}] \to 0$ as $c \to \infty$.
Let $F$ be the event that there is at least one upwards jump in the first {$2c$} steps.
    Then,
    \[ 
        \p(F) = 1 - (1-p)^{2c} = 2pc + O\left((pc)^2\right) = \frac2c + O(c^{-2}). 
    \]
    If $F$ does not occur, then $M^{(1)} = \tilde{M} = 0$.
    We therefore concentrate on the scenario that $F$ does apply.
    Note that the increments $B^{(2)}_i$ are independent of $F$.
    We have
    \[ S^{(1)}_j - \tilde{S}_j = \sum_{i=1}^j B^{(2)}_i . \]
    By Hoeffding's inequality,
   $
        \p(S^{(1)}_j - \tilde{S}_j > t) \leq {\e^{-t^2/(ja)}},
   $
    and so by a union bound,
    \[ 
        \p(M^{(1)} - \tilde{M} > c^{2/3}) \leq 2c \cdot  {\e^{-c^{1/3}/(2a)}}.
    \]
    Since it is always the case that $M^{(1)} - \tilde{M} \leq  {2ca}$, we have that
    \[ 
        \E[M^{(1)} - \tilde{M} \,|\, F\,] \leq  {4c^2a\e^{-c^{1/3}/(2a)}} + c^{2/3} = O(c^{2/3}).
    \]
    We thus obtain that
    $
        \E[M^{(1)} - \tilde{M}] = \E[M^{(1)} - \tilde{M} \,|\, F\,] \, \p(F) = O(c^{-1/3}),
   $ which vanishes as $c\to\infty$,
    as required.
\end{proof}


\section{Lognormal distributions that satisfy the dilation order}\label{lognormalDil}

\begin{proposition}
Suppose that $A$ and $B$ are lognormally distributed random variables, such that $\E[\ln A]\leq \E[\ln B]$ and $\Var(\ln A)\leq \Var(\ln B)$. Then $A\dleq B$.
\end{proposition}

\begin{proof}
Let $m_1=\E[\ln A]$, $m_2=\E[\ln B]$, $s_1^2=\Var(\ln A)$ and $s_2^2=\Var(\ln B)$. Let $Z$ and $Z'$ be two independent standard normal random variables, and let $\epsilon=\sqrt{s_2^2-s_1^2}$. 
Then
\begin{align*}
    A - \E A \dis \hat{X} &:= \exp(m_1+s_1Z)-\exp(m_1+s_1^2/2),\\
    B  - \E B \dis \hat{Y} &:= \exp(m_2+s_1Z+\epsilon Z')-\exp(m_2+(s_1^2+\epsilon^2)/2).
\end{align*}
Now note that
\begin{align*}
   \E[\hat{Y}|\hat{X}]= \E[\hat{Y}|Z] &= \E\left[\exp(m_2+s_1Z+\epsilon Z')-\exp(m_2+(s_1^2+\epsilon^2)/2)|Z\right]\\
&= \exp(m_2+\epsilon^2/2)\left(\exp(s_1Z)-\exp(s_1^2/2)\right) 
=\exp(m_2-m_1+\epsilon^2/2)\hat{X}.
\end{align*}
Thus by Lemma~\ref{condExp},
\begin{align*}
\exp(m_2-m_1+\epsilon^2/2)(A-\E A)\cleq B-\E B.
\end{align*}
(This is similar to the argument that $X'_i\cleq B_{i+1}-x_{i+1}$ in the proof of Lemma~\ref{maxLog}.)
Now we can apply Theorem 3.A.18 from \cite{Shaked}, that says that $X\dleq aX$ for $a\geq 1$, to see that
\begin{align*}
A-\E A\cleq \exp(m_2-m_1+\epsilon^2/2)(A-\E A)\cleq B-\E B.
\end{align*}
This proves that $A\dleq B$.
\end{proof}


\section{Numerical experiments for mean-based schedules}\label{MBNumerics}

{In this appendix we describe numerical experiments for exponential and lognormal service time distributions, in which the approximation ratio $\varrho_\omega$ is computed. The goal is to gain insight into what performance can be expected for the \svf{} sequence in practice. As the mean-based schedule is in place, (\ref{I}) implies that the total expected idle time is equal to $\E W_n$, the expected waiting time of the last patient.} The \svf{} rule is expected to be good for the patients in the beginning of the sequence, as then the previous patient is less variable, but might not be as good for patients at the end of the sequence, as they have high variance patients directly before them. Therefore, we expect the idle time to suffer more than the waiting times when the \svf{} sequence is used instead of the optimal sequence. This explains why we  set $\omega=1$ unless stated otherwise.

This choice of $\omega$ is further motivated by the results in \cite{Kong}. They constructed examples where \svf{} is not optimal when mean-based schedules are used, for an objective function that is a weighted average of expected waiting time and expected overtime. When waiting time and overtime had equal weight, typically more than 50 patients were needed for their counterexamples (refuting \svf{} being optimal, that is). When the overtime cost was much larger than the waiting time cost, much fewer patients were needed in creating counterexamples. 

The following proposition {reduces the complexity of finding the optimal sequence.}

\begin{proposition}\label{largestLast}
Under Assumption~$\ref{AStochDom}$ (dilation ordering assumption) and when mean-based schedules are used, there is an optimal sequence $\tau$ that satisfies $\tau(n)=n$, i.e.\ the largest variance should always be sequenced last.
\end{proposition}

\begin{proof}
Suppose $\tau$ is an optimal sequence. Note that $W_{k+1}=\max\{0,S_1,\dots,S_k\}$ is a convex function in each of the $X_i$. If $\tau(n)=n$ we are done. If $\tau(n)\neq n$, then under Assumption~\ref{AStochDom} we have $X_{\tau(n)}\leq_{\text{cx}} X_n$, and so switching the positions of patient $n$ and patient $\tau(n)$ can only decrease $\E W_{k+1}$. As this is true for all $k=1,\dots,n-1$, and the total idle time equals $\E W_n$, switching patients $\tau(n)$ and $n$ can only decrease the cost. The resulting sequence after the switch is therefore also optimal, and we are done. 
\end{proof}

This proposition still leaves $(n-1)!$ potential candidates for the optimal sequence. We are not aware of any other generally applicable structural result that can help to further reduce this number. As computing $\varrho_\omega$ thus requires checking all $(n-1)!$ potential candidates to find the optimal sequence, this severely limits the numbers of $n$ for which we can do this. 

\subsection{Exponential distributions}

We first look at the case where the service times are exponentially distributed, as this makes it possible to efficiently calculate the cost   for a given sequence and schedule \cite{Wang}. We set the service rate of patient $i$ to $n+1-i$, i.e., $\E B_i=1/(n+1-i)$. This way, the service rates are very different from one another, so that not using the optimal sequence could potentially have a large impact. The optimal sequence is found by complete enumeration over the $(n-1)!$ candidates. This allowed us to compute the approximation ratio $\varrho_1$ for problem instances with $n=3,\dots,11$ (within a reasonable amount of time). The results are given in Table \ref{MBOpt}. Choosing $\omega$ smaller than 1 resulted in smaller approximation ratios, as suspected, with \svf{} being optimal for all considered problem instances when $\omega=0.9$.

\begin{table}
\centering
\renewcommand{\arraystretch}{1.2}
\begin{tabular}{c|c|c|c|c}
$n$ & optimal sequence & optimal cost & cost for SVF & approximation ratio \\ \hline
3 & 1,2,3 & 0.2646 & 0.2646 & 1 \\
4 & 1,2,3,4 & 0.3098 & 0.3098 & 1 \\
5 & 2,1,3,4,5 & 0.3388 & 0.3389 & 1.0003 \\
6 & 3,1,2,4,5,6 & 0.3588 & 0.3590 & 1.0007 \\
7 & 4,2,1,3,5,6,7 & 0.3735 & 0.3739 & 1.0011 \\
8 & 5,3,1,2,4,6,7,8 & 0.3847 & 0.3853 & 1.0014 \\
9 & 6,4,2,1,3,5,7,8,9 & 0.3936 & 0.3943 & 1.0017 \\
10 & 7,5,3,1,2,4,6,8,9,10 & 0.4008 & 0.4015 & 1.0019 \\
11 & 8,5,3,1,2,4,6,7,9,10,11 & 0.4067 & 0.4076 & 1.0021 \\
\end{tabular}
\caption{The optimal sequence, optimal cost, cost of the SVF sequence, and approximation ratio $\varrho_1$ when patient $i$ has an exponentially distributed service time with rate $n+1-i$, and mean-based schedules are used.}\label{MBOpt}
\end{table}

For $n=9$ patients we also computed the approximation ratio for 100 random problem instances, where each service time was exponentially distributed with a service rate drawn independently from a uniform distribution on $[0,1]$. The largest approximation ratio found this way was 1.0034. For exponentially distributed service times it thus seems that the \svf{} rule will typically perform within $0.5\%$ of optimal.

From Table~\ref{MBOpt} we notice that the optimal sequence for each problem instance is ``V-shaped'': the variances first decrease, and later increase. 
We also checked that the optimal sequences for our 100 random problem instances were all V-shaped.
Restricting to V-shaped schedules allows us to consider substantially larger values of $n$.
For instances with service times $n+1-i$ for patient $i$, we obtained a bound of $1.0038$ for $n=23$, and for randomly chosen rates, found nothing worse than $1.0048$ for $n=17$.
However, we are aware of examples with other distributions for which the optimal schedule is not V-shaped, and so the true
value of $\rho_\omega$ may be slightly larger.

\subsection{Lognormal distributions}

Lognormal service times are intensively used in health care \cite{Cayirli,Klassen}. {For these distributions there is no clear efficient way to exactly compute the cost. We therefore work with its discrete counterpart, where service times are only allowed to take integer values.} The lognormal distribution for patient $i$ has parameters $m_i:=\E \ln B_i$ and $s_i^2:= \Var\ln B_i$. 
Motivated by the two lognormal distributions found to fit real data in \cite{Cayirli} (having $s_i$ equal to 0.3169 and 0.3480), in our experiments we  {have fixed} $s_i=0.33$. We set  {$m_i=\ln(50)+\ln(i)$}, where the 50 is just a scaling parameter meant to improve the discrete approximation. If the time unit would be minutes, this means we round to full minutes while all mean service times are over 50 minutes, in line with the accuracy one expects to be found in practice. 

\begin{table}
\centering
\renewcommand{\arraystretch}{1.2}
\begin{tabular}{c|c|c|c|c}
$n$ & optimal sequence & optimal cost & cost for SVF & approximation ratio \\ \hline
3 & 1,2,3 & 18.0131 & 18.0131 & 1 \\
4 & 2,1,3,4 & 32.6526 & 32.5799 & 1.0022 \\
5 & 3,1,2,4,5 & 50.3769 & 50.1484 & 1.0046 \\
6 & 4,2,1,3,5,6 & 70.8629 & 70.4453 & 1.0059 \\
7 & 5,3,1,2,4,6,7 & 93.8700 & 93.2174 & 1.0070 \\
8 & 6,4,2,1,3,5,7,8 & 119.210 & 118.302 & 1.0077 \\
9 & 7,5,3,1,2,4,6,8,9 & 146.730 & 145.538 & 1.0082 \\
10 & 8,5,3,1,2,4,6,7,9,10 & 176.305 & 174.803 & 1.0086 \\
\end{tabular}
\caption{The optimal sequence, optimal cost, cost of the SVF sequence, and approximation ratio $\varrho_1$ when patient $i$ has a service time that is a discrete approximation to a lognormal distribution with $m_i=\ln(50)+\ln(i)$ and $s_i=0.33$, and mean-based schedules are used.}\label{MBLognormal}
\end{table}

For these discrete random variables, we found the optimal sequence for problem instances with {$n=3,4,\dots,10$}. The results are reported in Table~\ref{MBLognormal}. For each of these problem instances the approximation ratio was below {1.01}. We also performed the same experiments for $s_i=0.75$ for all patients (instead of 0.33), for $n=3,\dots,10$. The approximation ratios were even smaller than in the $s_i=0.33$ case, with the largest approximation ratio being {1.0021}. When the same experiments were repeated for $\omega=0.9$, we found that the \svf{} sequence is optimal for all problem instances. The experiments are indicative of \svf{} performing quite well for lognormal distributions, but not as well as for exponential distributions.

To gain insight into the effect of the service-time distribution having a tail on either the left or the right, we also assessed the effect of replacing the service times $B_i$ by the flipped distribution $-B_i$. For $n=10$ and $s_i=0.33$ we then found an approximation ratio of $1.0345$, which, while still very close to $1$, is much larger than the $1.0086$ reported for the non-flipped lognormal distributions. 

We also performed experiments based on the lognormal distributions fitted to data in \cite{Cayirli},
distinguishing  between ``new" and ``return" patients. Lognormal distributions with parameters indicated in Table~\ref{tbl:cayirli} were found to yield a good fit to the data for these two groups. In line with \cite{Cayirli}, we consider a session consisting of 10 patients. For our numerical experiments we used discrete approximations to these two lognormal distributions, with service times being rounded to full minutes.
We computed the approximation ratio $\varrho_1$ for 11 problem instances, each with 10 patients in total with 0 up to 10 of them corresponding to  new patients (and the others to return patients). The largest approximation ratio we found was $\varrho_1=1.0023$ for 7 new patients and 3 return patients. We thus conclude that the \svf{} rule performs well for the practical distributions proposed in \cite{Cayirli}.


\section{Numerical experiments for optimally-spaced schedules}\label{OptNumerics}

Here we consider some numerical experiments when an optimally-spaced schedule is used for each sequence. Similar to Appendix~\ref{MBNumerics}, we will consider exponential service-time distributions and a discrete approximation to lognormal distributions.
In order to find the optimal schedule, we note that each waiting time is a convex function in each of the $x_i$, as by (\ref{W}) it is the maximum of linear functions in $x_i$. Note that the case $\omega=1$ is no longer interesting, as this results in an optimal schedule with $x_i=0$ for all $i$, resulting in a cost of zero for any sequence. Instead, we consider the parameter values $\omega=0.5$, $\omega=0.8$ and $\omega=0.9$.

Again we have, similar to Proposition~\ref{largestLast}, the result that the largest variance always should be sequenced last, as stated in the next proposition. This again leaves us with $(n-1)!$ candidates, and again it seems no more structural results are known, so we can only find the optimal sequence by enumerating over all $(n-1)!$ candidates.

\begin{proposition}
Under $Assumption~\ref{AStochDom}$ (dilation ordering assumption), there is an optimal combination of sequence $\tau$ and schedule $\bm{x}$ that satisfies $\tau(n)=n$, i.e.\ the largest variance should always be sequenced last.
\end{proposition}

\begin{proof}
Suppose $\tau$ and $\bm{x}$ are an optimal combination of sequence and schedule. If $\tau(n)=n$ we are done. If $\tau(n)\neq n$, then under Assumption~\ref{AStochDom} we have $B_{\tau(n)}-\E B_{\tau(n)}\leq_{\text{cx}} B_n-\E B_n$. We then propose a new sequence $\tau'$ with patients $\tau(n)$ and $n$ switched (so now $\tau'(n)=n$), and a new schedule $\bm{y}$ with $y_{\tau(n)} = x_n+\E B_{\tau(n)}-\E B_n$ and $y_i=x_i$ for $i\neq \tau(n)$. In the expression (\ref{W}) for $W_{k+1}$ this change effectively makes it so that each term $B_n-\E B_n$ gets replaced by $B_{\tau(n)}-\E B_{\tau(n)}$, and as $W_{k+1}$ is a convex function in this expression we see that this change cannot increase the $\E W_{k+1}$. By expression (\ref{I}) we see that the expected idle time only changes through the term $\E W_n$, and so the total cost will not be increased after the change. Therefore the sequence $\tau'$ with schedule $\bm{y}$ is also optimal, and as $\tau'(n)=n$ we are done.
\end{proof}

\subsection{Exponential distributions}

Suppose that, as  in Appendix~\ref{MBNumerics}, the service time of patient $i$ is exponentially distributed with rate $n+1-i$. The exponential service times allow us to efficiently compute the cost for any given schedule, and the convexity of the cost in each of the interarrival times $x_i$ then allows us to compute the optimal schedule. Due to the extra complexity involved in computing the optimal schedule for each possible sequence, we only computed the approximation ratio for $n=3,\dots,8$, for each of mentioned parameter values of $\omega$. For each of these problem instances, we found that the \svf{} sequence is optimal.

\subsection{Lognormal distributions}

For lognormal distributions, we will use the same discrete approximation as was used in Appendix~\ref{MBNumerics}. For patient $i$, the underlying lognormal distribution is set to have {$m_i:=\E \ln B_i =\ln(20)+\ln(i)$}, the {20} again being a scaling parameter.  
We need to restrict ourselves to schedules with integer-valued interarrival times in order to be able to effectively compute the cost function for a given schedule. Fortunately, this class of schedules is known to contain the optimal schedule, and the cost function of the resulting discrete optimization problem is L-convex \cite{Begen}. We can thus use the {\sc odicon} solver \cite{Tsuchimura}, that can minimize L-convex functions, to find the optimal schedule for each sequence.
{For $s_i=0.33$ ($s_i$ being the standard deviation of $\ln B_i$), we computed the approximation ratio for $n=3,\dots,7$, for each of the mentioned values of $\omega$. When $s_i=0.75$, we did the same for $n=3,\dots,6$.} For all considered problem instances, we found that the \svf{} sequence is optimal. 

\begin{remark}
    The method to find the optimal interarrival times described here, with the use of {\sc odicon}, works for any discrete service-time distribution. It was also used to find the optimal interarrival times in Example~\ref{OptimalCounterExample}.
\end{remark}

\end{document}